\documentclass[11pt]{article}

\usepackage{enumerate}
\usepackage{geometry} % 设置边距，符合Word设定
\usepackage{amsmath}
\usepackage{amsthm}
\usepackage{thmtools}
\usepackage{amssymb}
\usepackage{mathrsfs}
\usepackage{lipsum}
\usepackage{graphicx}
\graphicspath{{./figure/}}
\DeclareGraphicsExtensions{.pdf,.jpeg,.png,.jpg}
\usepackage{xcolor}
\usepackage{tikz}
\usepackage{tkz-euclide}

\usepackage{hyperref}
\usepackage[capitalize]{cleveref} %引用时可以显示Lemma等开头
\usepackage{titletoc}
\usepackage{float}
\usepackage{enumitem}

\setlist[enumerate]{align=left} %enumerate标签左对齐
\numberwithin{equation}{section}

\newtheorem{theorem}{Theorem}[section]
\newtheorem{lemma}[theorem]{Lemma}
\newtheorem{proposition}[theorem]{Proposition}

\theoremstyle{definition}
\newtheorem{remark}[theorem]{Remark}
\theoremstyle{definition}
\newtheorem{definition}[theorem]{Definition}
\newtheoremstyle{withcitation}
{\topsep}{\topsep}{\itshape}{}{\bfseries}{.}{ }{\thmname{#1}\thmnumber{ #2}\thmnote{ (#3)}}
\theoremstyle{withcitation}

\newcommand{\condition}[1]{%
  \phantomsection\label{cond:#1}% 隐藏的锚点
  \text{(#1)}% 显示条件标记
}

\newcommand{\refcond}[1]{%
  \hyperref[cond:#1]{(#1)}% 可点击的引用
}

\title{Ground state solutions of $p$-Laplacian equations with nonnegative potentials on Lattice graphs}
\author{Xinrong Zhao}
\date{}
\begin{document}

\maketitle

\begin{abstract}
    In this paper, we study the $p$-Laplacian equation 
    $$
    -\Delta_p u + V(x)|u|^{p-2}u = f(x,u)
    $$ 
    on the lattice graph $\mathbb{Z}^N$ with nonnegative potentials, where $\Delta_p$ is the discrete $p$-Laplacian and $p\in(1,\infty)$. By employing the Nehari manifold method, we establish the existence of ground state solutions under suitable growth conditions on the nonlinearity $f(x,u)$, provided that the potential $V(x)$ is either periodic or bounded. Moreover, we prove that if $f$ is odd in $u$ and $p\geq2$, then the above equation admits infinitely many geometrically distinct solutions. Finally, we extend these results from $\mathbb{Z}^N$ to the more general setting of Cayley graphs.
\end{abstract}

\section{Introduction}\label{sec:1}
The following $p$-Laplacian equation
\begin{equation}\label{eq:EquationFor_Continuous}
        -\Delta_p u + V(x)|u|^{p-2}u = f(x,u),\quad u\in W^{1,p}(\Omega),
\end{equation}
where $\Delta_p u:={\rm div}(|\nabla u|^{p-2}\nabla u)$ and $\Omega$ is a domain in $\mathbb{R}^N$, has attracted significant interest over the past several decades. In the special case $p=2$, equation \eqref{eq:EquationFor_Continuous} reduces to the classical semilinear Schrödinger equation, which has been extensively studied in the literature, see, for example, \cite{MR807816, MR1162728, MR2100913, MR2136816, Szulkin-Weth-paper, Szulkin2010method}. For the general case $p>1$, we refer readers to \cite{MR2037753, MR1856715, MR2371112, MR1998432, MR2472927} for results on bounded domains, and to \cite{MR2194501, MR2474596, MR2458640, MR2794422} for results on the entire space $\mathbb{R}^N$. 
It is well known that the solutions of \eqref{eq:EquationFor_Continuous} are precisely the critical points of the energy functional 
$$
\Phi(u)= \frac{1}{p}\int_{\mathbb{R}^N}|\nabla u|^p+V(x)|u|^p \,dx -\int_{\mathbb{R}^N}F(x,u)\, dx,
$$
where $F(x,t)=\int_0^{t}f(x,s)\, ds$. Consequently, variational methods provide a powerful framework for studying the solutions of \eqref{eq:EquationFor_Continuous}. 
The ground state solutions of \eqref{eq:EquationFor_Continuous}, defined as the solutions corresponding to the least positive critical value of $\Phi$, are of particular importance. 

For $p=2$, Li et al. \cite{Li-Wang-Zeng} established the existence of ground state solutions of equation \eqref{eq:EquationFor_Continuous} on $\mathbb{R}^N$ for two types of potentials $V(x)$: periodic potentials, and bounded potentials satisfying $\lim_{|x|\to\infty} V(x) = \sup_{\mathbb{R}^N} V(x) < +\infty$. Their approach is based on minimizing the functional $\Phi$ over the Nehari manifold $\mathcal{N}$ under suitable conditions on the nonlinearity $f$. Subsequently, Szulkin and Weth \cite{Szulkin-Weth-paper} generalized the Nehari manifold approach to the indefinite setting, proving the existence of ground state solutions and, additionally, infinitely many geometrically distinct solutions when $f$ is odd in $u$. Later, Liu \cite{LiuShibo} extended the results in \cite{Li-Wang-Zeng} to general $p>1$ using a mountain pass type argument and weakened the conditions on the nonlinearity $f$. 
%generalized the result on periodic potentials case to indefinite case

Recently, there has been growing interest in the study of various differential equations on graphs. For general graphs, several existence results have been established for the equation $-\Delta u+V(x)u=f(x,u)$ under the assumption that the potential function $V(x)$ tends to infinity as the combinatorial distance $d(x, x_0)\to\infty$ for some $x_0\in\mathbb{V}$ (see, e.g., Grigor’yan et al. \cite{MR3665801} and Zhang–Zhao \cite{MR3833747}).
When restricted to the lattice graphs $\mathbb{Z}^N$, more refined results are available. In particular, Hua and Xu \cite{Hua-XU} showed that the results of Li et al. \cite{Li-Wang-Zeng} in the continuous setting mentioned above can be extended to $\mathbb{Z}^N$. They further generalized their work to the general $p$-Laplacian equation $-\Delta_p u+V(x)|u|^{p-2}u=f(x,u)$ on $\mathbb{Z}^N$ for any $p\in(1,\infty)$ in \cite{Hua-Xu-p-Laplacian}. 

It is noteworthy that in all previous results regarding equation \eqref{eq:EquationFor_Continuous} on either $\mathbb{R}^N$ or $\mathbb{Z}^N$, the potential $V(x)$ is always assumed to have a positive lower bound. The main challenge in treating the case $\inf_{\mathbb{R}^N} V(x)=0$ stems from the failure of embedding the homogeneous Sobolev space $\mathcal{D}^{1,p}(\mathbb{R}^N)$ into $\ell^q(\mathbb{R}^N)$ for $q\neq p^*:=\frac{Np}{N-p}$. In the setting of $\mathbb{Z}^N$, the following discrete Sobolev inequality holds, see \cite[Theorem 3.6]{MR3397320} for proof:
$$
\|u\|_{\ell^{p^*}(\mathbb{Z}^N)}\leq C_p\|u\|_{\mathcal{D}^{1,p}(\mathbb{Z}^N)},\quad \forall u\in{\mathcal{D}^{1,p}(\mathbb{Z}^N)},
$$
where $N\geq2$, $p\in[1,N)$. Since $\ell^q(\mathbb{Z}^N)$ embeds into $\ell^r(\mathbb{Z}^N)$ for any $r > q\geq1$, see \cref{lem:LargeMoreThenSmall}, one verifies that for all $q\geq p^*$, the above discrete Sobolev inequality remains valid. That is, there exists a constant $C_{p,q}>0$ such that
\begin{equation}\label{eq:DiscreteSobolev}
    \|u\|_{\ell^q(\mathbb{Z}^N)}\leq C_{p,q}\|u\|_{\mathcal{D}^{1,p}(\mathbb{Z}^N)},\quad \forall u\in{\mathcal{D}^{1,p}(\mathbb{Z}^N)},
\end{equation}
where $N\geq2$, $p\in[1,N)$ and $q\geq p^*=\frac{Np}{N-p}$.
This implies that $\mathcal{D}^{1,p}(\mathbb{Z}^N)$ can indeed be embedded into $\ell^q(\mathbb{Z}^N)$ for all $q\geq p^*$. Consequently, it is possible to prove the existence of ground state solutions for the $p$-Laplacian equation \eqref{eq:EquationFor_Continuous} on $\mathbb{Z}^N$ in the case $\inf_{\mathbb{Z}^N} V(x)=0$, which constitutes the main focus of the present work.

In this paper, we consider the $p$-Laplacian equation
$$
-\Delta_p u + V(x)|u|^{p-2}u = f(x,u)
$$
on graphs. We begin by introducing some basic notations on graphs. Let $\mathbb{G} = (\mathbb{V}, \mathbb{E})$ be a simple, undirected, locally finite graph with vertex set $\mathbb{V}$ and edge set $\mathbb{E}$. For two vertices $x,y$, we write $x\sim y$ if there is an edge connecting them.

We denote the space of all real-valued functions on $\mathbb{G}$ by $C(\mathbb{V})$ and the subspace of functions with finite support by $C_c(\mathbb{V})$. The $p$-Laplacian on $\mathbb{G} = (\mathbb{V}, \mathbb{E})$ for a function $u\in C(\mathbb{V})$ is defined as
$$
\Delta_p u(x) := \sum_{y\in \mathbb{V},y\sim x} |u(y)-u(x)|^{p-2} (u(y) - u(x)),\ \forall x\in\mathbb{V}.
$$
The space $E(\mathbb{V})$, or simply $E$ when $\mathbb{V}$ is clear from the context, is defined as the completion of $C_c(\mathbb{V})$ with respect to the norm
\begin{equation}\label{eq:ExpressionOfNorm}
    \|u\|:= \left(\sum_{\substack{x,y\in\mathbb{V}\\x\sim y}}\frac{1}{2}|u(y)-u(x)|^p+\sum_{x\in\mathbb{V}}V(x)|u(x)|^p\right)^{1/p}.
\end{equation}
Note that when $V(x)\equiv0$, the space $E(\mathbb{V})$ reduces to the homogeneous Sobolev space $\mathcal{D}^{1,p}(\mathbb{V})$.

The integer lattice graph $\mathbb{Z}^N$ is defined as the graph with vertex set $\mathbb{V} = \mathbb{Z}^N$ and edge set
$$
\mathbb{E} = \left\{ \{x, y\} : x, y \in \mathbb{Z}^N, \sum_{i=1}^N |x_i - y_i| = 1 \right\}.
$$
For a given $T\in\mathbb{N}$, a function $h$ defined on $\mathbb{Z}^N$ is said to be $T$-periodic if it satisfies
$$
h(x + T e_i ) = h(x), \ \forall x \in \mathbb{Z}^N , 1 \leq i \leq N,
$$
where $e_i$ denotes the unit vector in the $i$-th coordinate.

We are concerned with the solutions of the following $p$-Laplacian type equation on the graph $\mathbb{G}$:
\begin{equation}\label{eq:MainEquation}
    \begin{cases}
        -\Delta_p u + V(x)|u|^{p-2}u = f(x,u), x \in \mathbb{V},\\
        u\in E(\mathbb{V}).
    \end{cases}
\end{equation}

\begin{theorem}\label{thm: 1}
Consider the $p$-Laplacian equation \eqref{eq:MainEquation} on $\mathbb{Z}^N$, $1<p<N$. Suppose that the functions $V:\mathbb{Z}^N\to \mathbb{R}$ and $f: \mathbb{Z}^N\times \mathbb{R}\to \mathbb{R}$ satisfy the following conditions:
\begin{enumerate}
    \item[\rm\condition{$S_1$}] $V(x)$ and $f(x,t)$ are both T-periodic in $x$ and $\inf_{x\in\mathbb{Z}^N}V(x)\geq 0$. 
    
    \item[\rm\condition{$S_2$}] $f(x,t)$ is continuous with respect to $t$ for every $x\in\mathbb{Z}^N$ and $f(x,t)\leq a(1+|t|^{q-1})$ for some constant $a>0$ and $q> p^*$.
%实际上这里给q加下界没有意义, 只是为了后面书写证明的方便.

    \item[\rm\condition{$S_3$}] $f(x,t)=o(|t|^{p^*-1})$ uniformly in $x$ as $|t|\to 0$, where $p^*=\frac{Np}{N-p}$.
    
    \item[\rm\condition{$S_4$}] $t \mapsto \frac{f(x,t)}{|t|^{p-1}}$ is strictly increasing on $(-\infty, 0) \cup (0, \infty)$ for each $x \in \mathbb{Z}^N$.
    
    \item[\rm\condition{$S_5$}]
    $\frac{F(x,t)}{|t|^p}\to\infty$ uniformly in $x$ as $|t|\to\infty$, where $F(x,t)=\int_{0}^{t}f(x,s) ds$.

\end{enumerate}
    Then equation \eqref{eq:MainEquation} admits a ground state solution. Moreover, if $F(x,t)\leq F(x,|t|)$ for all $(x,t)\in\mathbb{Z}^N\times\mathbb{R}$, then every ground state solution of \eqref{eq:MainEquation} is either positive or negative.
\end{theorem}

\begin{theorem}\label{thm: 2}
    Consider the $p$-Laplacian equation \eqref{eq:MainEquation} on $\mathbb{Z}^N$, $1<p<N$. Suppose that the functions $V:\mathbb{Z}^N\to \mathbb{R}$ and $f: \mathbb{Z}^N\times \mathbb{R}\to \mathbb{R}$ satisfy the conditions {\rm\refcond{$S_2$} }$\sim${\rm\refcond{$S_5$}}. Let $V_{-}:=(|V|-V)/{2}$ denotes the negative part of $V$. Assume further that
    \begin{enumerate}
    \item[\rm(i)] $ 0 \leq \sup_{x \in \mathbb{Z}^N} V(x) = V_{\infty} < \infty $, where $ V_{\infty} := \lim_{|x| \to \infty} V(x) $;
    \item[\rm(ii)] there exists some $ r \geq p^* $ such that
    $$
    \|V_{-}\|_{\ell^{\frac{r}{r-p}}(\mathbb{Z}^N)}< S_{p,r}^p,\quad
    \text{where}\quad 0<S_{p,r}:=\inf_{u\neq 0}\frac{\|u\|_{\mathcal{D}^{1,p}(\mathbb{Z}^N)}}{\|u\|_{\ell^r(\mathbb{Z}^N)}}.
    $$
    \end{enumerate}
    Then equation \eqref{eq:MainEquation} admits a ground state solution. Moreover, if $F(x,t)\leq F(x,|t|)$ for all $(x,t)\in\mathbb{Z}^N\times\mathbb{R}$, then every ground state solution of \eqref{eq:MainEquation} is either positive or negative.
\end{theorem}
\begin{remark}
This theorem allows the potential $ V(x) $ to have a nonzero negative part. The existence of a ground state solution is guaranteed provided that the negative part $ V_{-} $ is sufficiently small in the $ \ell^{\frac{r}{r-p}} $-norm for some $r\geq p^*$. This phenomenon contrasts with the continuous case, where the expression in \eqref{eq:ExpressionOfNorm} may fail to define a norm if $ V $ is not nonnegative. The result also underscores the role of the best Sobolev constant $ S_{p,r} $, as its value provides an admissible bound for the negativity of $ V $ in the existence criterion.
\end{remark}

When $f(x,u)=|u|^{q-2}u$, Hua and Li \cite{Hua-Li} proved that for all $q>p^*$, the best Sobolev constant 
\begin{equation}
    0<S_{p,q}=\inf_{u\in\mathcal{D}^{1,p}(\mathbb{Z}^N)\setminus\{0\}}\frac{\|u\|_{\mathcal{D}^{1,p}(\mathbb{Z}^N)}}{\|u\|_{\ell^q(\mathbb{Z}^N)}}
\end{equation}
is attained by some positive function $u\in\mathcal{D}^{1,p}(\mathbb{Z}^N)$, thereby proving the existence of positive solutions of the equation
\begin{equation}\label{eq:SimpleEq}
    -\Delta_p u(x)=|u(x)|^{q-2}u(x),\quad u\in\mathcal{D}^{1,p}(\mathbb{Z}^N).
\end{equation}
Note that \Cref{thm: 1} also ensures the existence of positive solutions to \eqref{eq:SimpleEq}. In fact, up to a positive scaling factor, the ground state solutions of equation \eqref{eq:SimpleEq} are precisely the functions that attain the best Sobolev constant $S_{p,q}$.

\begin{theorem}\label{thm:|u|p}
    Assume that $1<p<N$ and $q>p^*$, then every ground state solution of equation \eqref{eq:SimpleEq} attains the best Sobolev constant
    $$
    S_{p,q}=\inf_{u\in\mathcal{D}^{1,p}(\mathbb{Z}^N)\setminus\{0\}}\frac{\|u\|_{\mathcal{D}^{1,p}(\mathbb{Z}^N)}}{\|u\|_{\ell^q(\mathbb{Z}^N)}}.
    $$
    Conversely, if $u\neq0$ attains the best Sobolev constant, then there exists a positive number $t_u$ such that $t_u u$ is a ground state solution of equation \eqref{eq:SimpleEq}.
\end{theorem}
%要不要把b和Spq的关系写一下.(但这个时候没有Nehari流形定义比较麻烦)

%继续

%Definition of geometrically distinct solutions
Next, we consider the multiplicity of solutions of \eqref{eq:MainEquation} on $\mathbb{Z}^N$. The analogous problem in the continuous case has been studied, for example, in \cite{Szulkin-Weth-paper}. For each $u\in E(\mathbb{Z}^N)$, denote by $\mathcal{O}(u)$ the orbit of $u$ under the action of $\mathbb{Z}^N$, that is, $\mathcal{O}(u):=\{u(\cdot-kT):k\in\mathbb{Z}^N\}$. Note that if $u_0$ is a solution of \eqref{eq:MainEquation}, then every element of $\mathcal{O}(u_0)$ is also a solution. We say that two solutions $u_1$ and $u_2$ of \eqref{eq:MainEquation} are geometrically distinct if their orbits $\mathcal{O}(u_1)$ and $\mathcal{O}(u_2)$ are disjoint.

\begin{theorem}\label{thm:InfiniteSolution}
     Suppose that $p\geq2$ and that $f(x,u)$ is odd in $u$ for every $x\in\mathbb{Z}^N$. Then, under the same assumptions as in \cref{thm: 1}, equation \eqref{eq:MainEquation} admits infinitely many geometrically distinct solutions.
\end{theorem}
\begin{remark}
    To the best of our knowledge, the case of $p>2$ in the continuous setting, under the assumption $\inf_{x}V(x)>0$, has not been explicitly established in the literature. It is noteworthy that the proof technique for \cref{lem: DiscretenessOfPS-sequences} can be adapted to the continuous case. Subsequently, by following the arguments presented in \cite{Szulkin-Weth-paper}, one can establish the corresponding multiplicity result for $p>2$ in the continuous setting as well.
\end{remark}

The preceding results can be extended to the setting of Cayley graphs. Let $X(\Gamma, S)$ be the Cayley graph of a discrete group ${\Gamma}$ with a finite symmetric generating set $S$. For any $r \in \mathbb{N}$, we define $V(r)$ as the number of group elements with word length at most $r$. It is well known that if $V(r) \geq C r^D$, $\forall r \geq 1$ for $D \geq 3$, then the Sobolev inequality
\begin{equation}\label{eq: SobolevOnCayley}
\|u\|_{\ell^q(\Gamma)} \leq C_{p,q} \|u\|_{\mathcal{D}^{1,p}(\Gamma)}
\end{equation}
holds, where $1 \leq p < D$, $q \geq \frac{Dp}{D - p}$. In fact, this follows from \cite[Theorem 3.6]{MR3397320} and the isoperimetric estimate in \cite[Theorem 4.18]{woess2000random}. Note that the action of $\mathrm{Aut}(\Gamma)$ on $\Gamma$ is transitive, where $\mathrm{Aut}(\Gamma)$ is the set of graph automorphisms of $\Gamma$.
By an analogous argument, we can prove the following result.

\begin{theorem}\label{thm:Cayley}
Let $X(\Gamma, S)$ be a Cayley graph with a finite generating set $S$, and suppose its volume growth satisfies $V(r) \geq C r^D$ for all $r \geq 1$ and some $D \geq 2$. Let $H$ be a subgroup of $\mathrm{Aut}(\Gamma)$ whose action on $\Gamma$ has finitely many orbits, and let $|x|_S$ denote the word length of $x \in \Gamma$ with respect to $S$. Given $1<p<D$,
\begin{enumerate}
    \item[\rm(1)] If we replace the Euclidean distance $|x|$ with word length $|x|_S$, then the conclusion of \cref{thm: 2} holds on $X(\Gamma, S)$.
    \item[\rm(2)] Furthermore, If the $T$-periodicity in condition {\rm \refcond{$S_1$}} is replaced by $H$-invariance of $V$ and $f$, then the conclusions of \cref{thm: 1} and \cref{thm:InfiniteSolution} also hold on $X(\Gamma, S)$.
\end{enumerate}
\end{theorem}
\begin{remark}
    It follows from \cref{thm:Cayley} that the $T$-periodicity in \refcond{$S_1$} can be generalized to $\vec{T}$-periodicity for some vector $\vec{T}=(T_1,\cdots,T_n)\in\mathbb{Z}^N_{+}$, i.e., $V(x)$ and $f(x,u)$ are $T_i$-periodic in $x_i$ for every $1\leq i\leq N$.
\end{remark}

This paper is organized as follows. In \cref{sec:2}, we recall the setting of function spaces on graphs and then introduce the variational framework based on the Nehari manifold. The proofs of \cref{thm: 1}, \cref{thm: 2} and \cref{thm:InfiniteSolution} are given in \cref{sec:3}, \cref{sec:4} and \cref{sec:5} respectively. The special case $f(x,u)=|u|^{q-2}u$ is considered in \cref{sec:4} and a brief proof of \cref{thm:Cayley} is given at the end of \cref{sec:5}. Finally, a detailed proof of \cref{lem: DiscretenessOfPS-sequences} is provided in the Appendix.

%In \cref{sec:3}, we give a proof of \cref{thm: 1} and then consider the special case $f(x,u)=|u|^{q-2}u$. In \cref{sec:3}, we first prove \cref{thm: 2} and then 

\section{Notations and Preliminaries}\label{sec:2}

\subsection{Notations}
We begin by defining several function spaces on $\mathbb{Z}^N$. Let $q\geq1$. For a function $u\in C(\mathbb{Z}^N)$, it's $\ell^q$-norm is defined by
$$
\|u\|_{\ell^q(\mathbb{Z}^N)} := 
\begin{cases}
\left(\sum_{x\in\mathbb{Z}^N} |u(x)|^{q}\right)^{1/q}, & 0 < q < \infty, \\
\sup_{x\in\mathbb{Z}^N} |u(x)|, & q = \infty.
\end{cases}
$$
The associated $\ell^q(\mathbb{Z}^N)$ space is
$$
\ell^q(\mathbb{Z}^N) := \left\{ u \in C(\mathbb{Z}^N) : \|u\|_{\ell^q(\mathbb{Z}^N)} < \infty \right\}.
$$
In this paper, we will write $\|u\|_q$ instead of $\|u\|_{\ell^q(\mathbb{Z}^N)}$ when no confusion arises.

The following lemma is well-known, see \cite[Lemma 2.1]{Huang-Li-Yin}.
\begin{lemma}\label{lem:LargeMoreThenSmall}
Suppose that $u \in \ell^q(\mathbb{Z}^N)$ for some $1\leq q\leq\infty$, then
$$
\|u\|_r \leq \|u\|_q, \quad \forall r \geq q.
$$
\end{lemma}

For any function $u \in C(\mathbb{Z}^N)$ and any adjacent vertices $x \sim y$, we define the difference operator
$$
\nabla_{xy}u := u(y) - u(x).
$$
The $\|\cdot\|$-norm introduced in \cref{sec:1} is then given by
$$
\|u\|= \left(\sum_{\substack{x,y\in\mathbb{Z}^N\\x\sim y}}\frac{1}{2}|\nabla_{xy}u|^p+\sum_{x\in\mathbb{Z}^N}V(x)|u(x)|^p\right)^{1/p}.
$$
Recall that the space $E$ is defined as the completion of $C_c(\mathbb{Z}^N)$ with respect to the norm $\|\cdot\|$. When $V(x)\equiv0$, $E$ coincides with the homogeneous Sobolev space $\mathcal{D}^{1,p}(\mathbb{Z}^N)$. In general, we have the continuous embedding $E\hookrightarrow\mathcal{D}^{1,p}(\mathbb{Z}^N)$. Therefore, by the discrete Sobolev inequality \eqref{eq:DiscreteSobolev}, $E$ can be continuously embedded 
into $\ell^q(\mathbb{Z}^N)$ for all $q \geq p^*$.

The energy functional associated with equation \eqref{eq:MainEquation} is defined by
\begin{equation}\label{eq: ExpressionOfPhi}
\begin{aligned}
    \Phi(u)&:=\frac{1}{p}\left(\sum_{\substack{x,y\in\mathbb{Z}^N\\x\sim y}}\frac{1}{2}|\nabla_{xy}u|^p+\sum_{x\in\mathbb{Z}^N}V(x)|u(x)|^p\right)-\sum_{x\in\mathbb{Z}^N}F(x,u(x))\\
    &=\frac{1}{p}\Vert u \Vert^p-\sum_{x\in\mathbb{Z}^N}F(x,u(x))
\end{aligned}
\end{equation}
% x下面这一步用到了小于等于 u的q次方.
The growth conditions on $f$, combined with the discrete Sobolev inequality, ensure that the functional $\Phi$ is well-defined on $E$. Moreover, one can show that $\Phi(\cdot)\in C^1(E,\mathbb{R})$ with the derivative $\Phi':E\to E^*$ given by
\begin{equation}\label{eq: DerivativeOfPhi}
    \langle \Phi'(u),v\rangle=\sum_{\substack{x,y\in\mathbb{Z}^N\\x\sim y}}\frac{1}{2}|\nabla_{xy}u|^{p-2}(\nabla_{xy}u)(\nabla_{xy}v)+\sum_{x\in\mathbb{Z}^N}V(x)|u|^{p-2}uv-\sum_{x\in\mathbb{Z}^N}f(x,u)v
\end{equation}
for all $v\in E$, where $\langle\cdot,\cdot\rangle$ denotes the duality pairing between $E^*$ and $E$. In particular, we have
\begin{equation}\label{eq: ExpressionOf<Phi'u,u>}
    \langle \Phi'(u),u\rangle=\|u\|^p-\sum_{x\in\mathbb{Z}^N}f(x,u)u(x)
\end{equation}

\subsection{Nehari manifolds}
In this subsection, we introduce the variational framework of Nehari manifolds. 
\begin{definition}
    The Nehari manifold $\mathcal{N}$ associated with the functional $\Phi$ is defined by
    $$
    \mathcal{N}:=\{u\in E\setminus\{0\}: \ \langle\Phi'(u),u\rangle=0\}.
    $$
    A function $u\in\mathcal{N}$ is called a ground state solution of equation \eqref{eq:MainEquation} if it satisfies 
    $$
    \Phi(u)=\inf_{v\in\mathcal{N}}\Phi(v).
    $$
\end{definition}
Note that from the definition of the Nehari manifold $\mathcal{N}$, it follows immediately that  for all $u\in\mathcal{N}$,
\begin{equation}\label{eq: PhiForNehariManifold}
    \Phi(u)=\Phi(u)-\frac{1}{p}\langle\Phi'(u),u\rangle=\frac{1}{p}\sum_{x\in\mathbb{Z}^N}f(x,u)u-\sum_{x\in\mathbb{Z}^N}F(x,u).
\end{equation}
\begin{lemma}\label{lem: UniquenessOft_u_G}
    Suppose that the conditions {\rm\refcond{$S_3$}$\sim$\refcond{$S_5$}} hold. Then
    \begin{enumerate}
        \item[\rm (1)] For all $u\in E$ and $t>0$, we have
    \begin{equation}\label{eq:Phi(u)AndPhi(tu)_G}
        \Phi(u)\geq \Phi(tu)+\frac{1-t^p}{p}\langle\Phi'(u),u\rangle.
    \end{equation}
    The equality holds if and only if  $(1-t)u\equiv0$.

    \item[\rm (2)] For every $ u \in E \setminus \{0\} $, there exists a unique $ t_u > 0 $ such that $ t_u u \in \mathcal{N} $. The function $ \psi_u(t) := \Phi(tu) $, defined for $ t > 0 $, attains its global maximum at $ t = t_u $. Moreover, $\psi_u(t)$ is strictly increasing on $ (0, t_u) $ and strictly decreasing on $ (t_u, +\infty) $.
    \end{enumerate}

\end{lemma}
\begin{proof}
(1). Let $u\in E$, a direct computation shows that
    $$
    \begin{aligned}
        \Phi(u)&=\frac{1}{p}\|u\|^p-\sum_{x\in\mathbb{Z}^N} F(x,u)\\
        &=\Phi(tu)+\frac{1-t^p}{p}\langle\Phi'(u),u\rangle+\sum_{x\in\mathbb{Z}^N} \left(\frac{1-t^p}{p}f(x,u)u+F(x,tu)-F(x,u)\right).
    \end{aligned}
    $$
    Differentiating the expression inside the sum with respect to $t$ yields
    $$
    \frac{d}{dt}\left(\frac{1-t^p}{p}f(x,u)u+F(x,tu)-F(x,u)\right)=t^{p-1}|u|^p\left(\frac{f(x,tu)}{|tu|^{p-1}}-\frac{f(x,u)}{|u|^{p-1}}\right).
    $$
    By \refcond{$S_4$}, the function $s\mapsto f(x,s)/|s|^{p-1}$ is strictly increasing on $(-\infty,0)\cup(0,\infty)$. Consequently, the derivative above is negative on $(0,1)$ and positive on $(1,\infty)$. This implies that for each fixed $x\in\mathbb{Z}^N$, the inequality
    \begin{equation}\label{eq: pF_Less_fu}
        \frac{1-t^p}{p}f(x,u)u+F(x,tu)-F(x,u)\geq0
    \end{equation}
    holds for all $t>0$ and the equality is attained if and only if $(1-t)u(x)=0$. Summing over $x$, we conclude that the inequality \eqref{eq:Phi(u)AndPhi(tu)_G} holds and the equality is attained if and only if $(1-t)u\equiv 0$.
    
(2). Fix $u\in E\backslash\{0\}$. Since
    $$
    \frac{d}{dt}\psi_u(t)=\frac{d}{dt}\Phi(tu)=\langle \Phi'(tu),u\rangle=\frac{1}{t}\langle \Phi'(tu),tu\rangle,
    $$
    the function $tu\in\mathcal{N}$ if and only if $t$ is a critical point of $\psi_u$. By conditions \refcond{$S_3$} and \refcond{$S_5$}, we have $\psi_u(0)=0$, $\psi_u(t)>0$ for sufficiently small $t$ and $\psi_u<0$ for sufficiently large $t$. It follows that there exists $t_u$ such that $\psi_u(t)$ attains its maximum at $t_u$, and therefore $t_uu\in\mathcal{N}$. 
    
    Now we only need to prove that $t_u$ is the unique critical point. Suppose that $t_0u\in\mathcal{N}$ for some $t_0>0$. Since $t_0u\in\mathcal{N}$, after applying \eqref{eq:Phi(u)AndPhi(tu)_G} with $u$ replaced by $t_0u$ and $t_u u$ respectively, we obtain that $\Phi(t_0u)\geq \Phi(tu)$ and $\Phi(t_uu)\geq \Phi(tu)$ for all $t>0$. This implies that $\Phi(t_0u)=\Phi(t_uu)$. It then follows from the equality in part (1) that $t_0=t_u$.
   
% Further we can see $t_u$ is continuous for $u\in E\setminus\{0\}$.
\end{proof}

We now present a general result concerning the Nehari manifold. A detailed proof can be found in \cite[Chapter 3]{Szulkin2010method}. For the remainder of this subsection, we assume that $E$ is a uniformly convex real Banach space, $S$ denotes the unit sphere of $E$, and $\Phi \in C^1(E, \mathbb{R})$ satisfies $\Phi(0) = 0$. A sequence $\{u_n\}_n \subset E$ is called a Palais-Smale sequence of $\Phi$ if $\{\Phi(u_n)\}_n$ is bounded and $\Phi'(u_n)\to0$ as $n\to\infty$.
A function $\varphi \in C(\mathbb{R}_+, \mathbb{R}_+)$ is said to be a {normalization function} if $\varphi(0) = 0$, $\varphi$ is strictly increasing and $\varphi(t) \to \infty$ as $t \to \infty$. We need the following further assumptions:

\begin{enumerate}
    \item[\condition{$A_1$}] There exists a normalization function $\varphi$ such that
    $$
    u \mapsto \psi(u) := \int_0^{\|u\|} \varphi(t) \, dt \in C^1(E \setminus \{0\}, \mathbb{R}),
    $$
    $J := \psi'$ is bounded on bounded sets and $J(w)w = 1$ for all $w \in S$.
    
    \item[\condition{$A_2$}] For each $w \in E \setminus \{0\}$ there exists $s_w>0$ such that if $\alpha_w(s) := \Phi(sw)$, then $\alpha_w'(s) > 0$ for $0 < s < s_w$ and $\alpha_w'(s) < 0$ for $s > s_w$.
    
    \item[\condition{$A_3$}] There exists $\delta > 0$ such that $s_w \geq \delta$ for all $w \in S$ and for each compact subset $W \subset S$ there exists a constant $C_W$ such that $s_w \leq C_W$ for all $w \in W$.
\end{enumerate}

\begin{remark}\label{rmk:A1}
    In our setting $E=E(\mathbb{V})$, we take the normalization function $\varphi(t) := t^{p-1}$. The associated functional is then given by $\psi(u) = \frac{1}{p} \|u\|^p$, and its derivative $J = \psi' : E \to E^*$ is given by 
    \begin{equation*}
    \langle J(u),v\rangle = \sum_{\substack{x,y\in\mathbb{Z}^N\\x\sim y}}\frac{1}{2}|\nabla_{xy}u|^{p-2}(\nabla_{xy}u)(\nabla_{xy}v)+\sum_{x\in\mathbb{Z}^N}V(x)|u(x)|^{p-2}u(x)v(x), \ \forall v\in E.
    \end{equation*}
    In particular, $\langle J(u),u\rangle=\|u\|^p=1$ for all $u\in S$. By the H$\ddot{\rm o}$lder inequality, we have $\langle J(u),v\rangle\leq\|u\|^{p-1}\|v\|$, which implies that $J(u)$ is bounded on bounded sets. Therefore, {\rm\refcond{$A_1$}} is satisfied.
\end{remark}

It follows from {\rm\refcond{$A_1$}} that $S$ is a $C^1$-submanifold of $E$ and the tangent space of $S$ at $w$ is given by
$$
T_w(S) = \{z \in E : \langle J(w),z\rangle = 0\}.
$$
In this general setting, the Nehari manifold is defined as
$$
\mathcal{N} := \{u \in E \setminus \{0\} : \langle \Phi'(u),u\rangle = 0\}.
$$
Assumption {\rm\refcond{$A_2$}} implies that $sw \in \mathcal{N}$ if and only if $s = s_w$, while the first part of assumption {\rm\refcond{$A_3$}} ensures that $\mathcal{N}$ is closed in $E$ and bounded away from $0$. 

Define the mappings $\hat{m} : E \setminus \{0\} \to \mathcal{N}$ and $m : S \to \mathcal{N}$ by
\begin{equation}
\label{eq:mappings}
\hat{m}(w) := s_w w \quad \text{and} \quad m := \hat{m}|_S.
\end{equation}

\begin{proposition}\label{prop: mIsHomeomorphism}
    Suppose that $\Phi$ satisfies {\rm\refcond{$A_2$}} and {\rm\refcond{$A_3$}}. Then
    
    {\rm(a)}. The mapping $\hat{m}$ is continuous.
    
    {\rm(b)}. The mapping $m$ is a homeomorphism between $S$ and $\mathcal{N}$ with its inverse given by $m^{-1}(u) = {u}/{\|u\|}$.
\end{proposition}

We now define the functionals $\hat{\Psi} : E \setminus \{0\} \to \mathbb{R}$ and $\Psi : S \to \mathbb{R}$ by
\begin{equation}
\label{eq:Psi_def}
\hat{\Psi}(w) := \Phi(\hat{m}(w)) \quad \text{and} \quad \Psi := \hat{\Psi}|_S.
\end{equation}
Note that although $\mathcal{N}$ is not necessarily a $C^1$-manifold, the functional $\hat{\Psi}$ is of class $C^1$ and there is a one-to-one correspondence between the critical points of $\Psi$ and the nontrivial critical points of $\Phi$. More precisely, we have the following result.

\begin{proposition}\label{prop: PSforPhiAndPsi}
    Suppose that $E$ is a Banach space satisfying {\rm\refcond{$A_1$}} and that $\Phi$ satisfies {\rm\refcond{$A_2$}} and {\rm\refcond{$A_3$}}. Then
\begin{enumerate}
    \item[\rm (a)] $\Psi \in C^1(S, \mathbb{R})$ and its derivative is given by
    $$
    \langle\Psi'(w),z\rangle = \|m(w)\| \langle\Phi'(m(w)), z\rangle \quad \text{for all } z \in T_w(S).
    $$
    \item[\rm (b)] If $\{w_n\}_n$ is a Palais-Smale sequence for $\Psi$, then $\{m(w_n)\}_n$ is a Palais-Smale sequence for $\Phi$. Conversely, if $\{u_n\}_n \subset \mathcal{N}$ is a bounded Palais-Smale sequence for $\Phi$, then $\{m^{-1}(u_n)\}_n$ is a Palais-Smale sequence for $\Psi$.
    \item[\rm (c)] $w\in S$ is a critical point of $\Psi$ if and only if $u=m(w)$ is a nontrivial critical point of $\Phi$. Moreover, $\Psi(w)=\Phi(u)$ and 
    $$
    \inf_{w\in S} \Psi(w) = \inf_{u\in\mathcal{N}} \Phi(u).
    $$
\end{enumerate}
\end{proposition}

\section{Proof of \cref{thm: 1}}\label{sec:3}
Throughout this section, conditions \refcond{$S_1$}$\sim$\refcond{$S_5$} are always assumed. We begin by establishing several preliminary results.
\begin{lemma}\label{lem: fu>pF}
    $f(x,s)s\geq pF(x,s)$ for all $x\in\mathbb{Z}^N$ and $s\in \mathbb{R}$.  
\end{lemma}
\begin{proof}
    Fix $x\in\mathbb{Z}^N$ and $s\in\mathbb{R}$. Define the function $g:(0,\infty)\to \mathbb{R}$ by 
    $$
    g(t):=\frac{t^p}{p}f(x,s)s-F(x,ts).
    $$
    Differentiating $g$ with respect to $t$ yields
    $$
    g'(t)=t^{p-1}|s|^p\left(\frac{f(x,ts)}{|ts|^{p-1}}-\frac{f(x,s)}{|s|^{p-1}}\right).
    $$
    By \refcond{$S_4$}, the function $s\mapsto f(x,s)/|s|^{p-1}$ is strictly increasing on $(-\infty,0)\cup(0,\infty)$. Consequently, $g'(t)<0$ for $t\in(0,1)$ and $g'(t)>0$ for $t\in(1,\infty)$. This implies that $g(1)\geq g(t)$ for all $t>0$. It follows that
    $$
    \frac{1}{p}f(x,s)s-F(x,s)=g(1)\geq \lim_{t\to 0^+} g(t)=0,
    $$
    which completes the proof.
\end{proof}

\begin{lemma}\label{lem: Phi_BoundedBelow}
    The functional $\Phi(\cdot)$ has a positive lower bound on $\mathcal{N}$, i.e.,
    $$
    b:=\inf_{u\in\mathcal{N}}\Phi(u)>0.
    $$
    In particular, it follows that 
    $$
    \inf\limits_{u\in\mathcal{N}}\|u\|\geq \sqrt[p]{pb}>0.
    $$
\end{lemma}
\begin{proof}
    By \cref{lem: UniquenessOft_u_G}, conditions \refcond{$S_2$}$\sim$\refcond{$S_3$} and the discrete Sobolev inequality \eqref{eq:DiscreteSobolev}, we obtain that for all $t>0$ and $u\in E$,
    \begin{equation}\label{eq:PositiveInf}
    \begin{aligned}
        \Phi(u)\geq \Phi(tu)&=\frac{t^p}{p}\|u\|^p-\sum_{x\in\mathbb{Z}^N} F(x,tu(x))\\
        &\geq \frac{t^p}{p}\|u\|^p-\sum_{x\in\mathbb{Z}^N} C( t^{p^*}|u(x)|^{p^*}+ t^{q}|u(x)|^{q})\\
        &\geq \frac{t^p}{p}\|u\|^p-C'(t^{p^*}\|u\|^{p^*}+t^{q}\|u\|^{q}),
    \end{aligned}
    \end{equation}
    where $C,C'>0$ are constants. Since $p^*,q>p$, it follows that $b_r:=\inf_{\|u\|=r}\Phi(u)>0$ for sufficiently small $r>0$. Moreover, by \cref{lem: UniquenessOft_u_G}, for all $u\in\mathcal{N}$ we have
    $$
    \Phi(u)\geq \Phi(\frac{ru}{\|u\|})\geq b_r>0,
    $$
    which implies that $b\geq b_r>0$.
    In particular, since $F(x,t)\geq0$ by \refcond{$S_4$},
    $$
    \inf_{u\in\mathcal{N}}\|u\|\geq\inf_{u\in\mathcal{N}} \sqrt[p]{p\Phi(u)}\geq\sqrt[p]{pb}>0.
    $$
\end{proof}

\begin{proposition}\label{prop: PhiCoercive}
    The functional $\Phi$ is coercive on $\mathcal{N}$, i.e., $\Phi(u_n)\to\infty$ as $n\to \infty$ whenever $\{u_n\}_n\subset\mathcal{N}$ and $\|u_n\|\to\infty$ as $n\to \infty$.
\end{proposition}
\begin{proof}
    Suppose, for contradiction, that there exists a sequence $\{u_n\}_{n}\subset\mathcal{N}$ and a constant $M>0$ such that $\|u_n\|\to\infty$ as $n\to \infty$ and $\sup_{n\in\mathbb{N}}|\Phi(u_n)|\leq M$. Define the normalizing sequence $v_n:=\frac{u_n}{\|u_n\|}$.
    
    \textbf{Case 1.} $\|v_n\|_q\to 0$ as $n\to\infty$. Fix $R > \sqrt[p]{2pM}$ and choose $\epsilon = \frac{M}{(RS_{p,p^*})^{p^*}}> 0$. By \refcond{$S_2$} and \refcond{$S_3$}, there exists a constant $C_{\epsilon}>0$ such that for all $x\in\mathbb{Z}^N$ and $t\in\mathbb{R}$,
    $$
    |F(x,t)|\leq \epsilon |t|^{p^*}+ C_{\epsilon} |t|^{q}.
    $$
    It then follows from the assumption $\|v_n\|_q\to 0$ and the discrete Sobolev inequality \eqref{eq:DiscreteSobolev} that
    \begin{equation*}
        \varlimsup_{n\to\infty}\sum_{x\in\mathbb{Z}^N} F(x,Rv_n)\leq  \epsilon R^{p^*}\varlimsup_{n\to\infty}\|v_n\|_{p^*}^{p^*}+C_{\epsilon}R^{q}\varlimsup_{n\to\infty}\|v_n\|_{q}^{q}\leq \epsilon R^{p^*}S_{p,p^*}^{p^*}=M.
    \end{equation*}
    Therefore, by \cref{lem: UniquenessOft_u_G} we have
    $$
    \begin{aligned}
         M\geq\varlimsup_{n\to\infty}\Phi(u_n)\geq \varlimsup_{n\to\infty}\Phi(Rv_n)
         =&\frac{R^p}{p}-\varliminf_{n\to\infty}\sum_{x\in\mathbb{Z}^N} F(x,Rv_n)\\
         \geq&\frac{R^p}{p}-M>M,
    \end{aligned}
    $$
    which is a contradiction.

    \textbf{Case 2.} $\varlimsup\limits_{n\to\infty}\|{v}_n\|_q>0$. Since $\|v_n\|=1$ for all $n\in\mathbb{N}$, it follows from the discrete Sobolev inequality that
    $$
        0<\varlimsup_{n\to\infty}\|{v}_n\|_q^q\leq\varlimsup_{n\to\infty}\|{v}_n\|_{\infty}^{q-p^*}\|{v}_n\|_{p^*}^{p^*}\leq S_{p,p^*}^{p^*}\varlimsup_{n\to\infty}\|{v}_n\|_{\infty}^{q-p^*},
    $$
    which implies that $\delta:=\varlimsup\limits_{n\to\infty}\|{v}_n\|_{\infty}>0$. Passing to a subsequence, we may assume that there exists a sequence $\{x_n\}_n\subset\mathbb{Z}^N$ such that $|{v}_n(x_n)|\geq\delta/2$ for all $n\in \mathbb{N}$. Let $\Omega:=[0,T)^N\cap\mathbb{Z}^N$. Then there exists $\{k_n\}_n\subset\mathbb{Z}^N$ such that $x_n-k_nT\in\Omega$ for all $n\in\mathbb{N}$. Define the translated sequences
    $$
    \tilde{u}_n(x):=u_n(x-k_nT),\ \tilde{v}_n(x):=\frac{\tilde{u}_n(x)}{\|\tilde{u}_n\|}=v_n(x-k_nT).
    $$
    Passing to a further subsequence, we may assume that there exists $x_0\in\Omega$ such that $\tilde{v}_n(x_0)\to \delta'$ with $\delta'\geq \frac{\delta}{2|\Omega|}$.
    Hence
    $$
    \tilde{u}_n(x_0)=\|\tilde{u}_n\|\tilde{v}_n(x_0)\to\infty.
    $$ 
    By \refcond{$S_5$} and the $T$-periodicity of functions $V(x)$ and $f(x,u)$, we obtain
    $$
    \begin{aligned}
    0 \leq \frac{\Phi(u_n)}{\|u_n\|^p}
    = \frac{\Phi(\tilde{u}_n)}{\|\tilde{u}_n\|^p}
    &= \frac{1}{p}
       - \sum_{x\in\mathbb{Z}^N}
         \frac{F(x, \tilde{u}_n)}{|\tilde{u}_n(x)|^p} |\tilde{v}_n(x)|^p \\
    &\leq \frac{1}{p}
       - \frac{F(x_0, \tilde{u}_n)}{|\tilde{u}_n(x_0)|^p} |\tilde{v}_n(x_0)|^p
       \to -\infty, \quad n \to \infty,
    \end{aligned}
    $$
    which is again a contradiction. 
\end{proof}

We now begin the proof.
\begin{proof}[Proof of \cref{thm: 1}]
    By \cref{lem: Phi_BoundedBelow}, we have $b:=\inf_{u\in\mathcal{N}}\Phi(u)>0$. A sequence $\{u_n\}_{n}\subset\mathcal{N}$ is called a minimizing sequence for $\Phi|_{\mathcal{N}}$ if $\Phi(u_n)\to b$ as $n\to\infty$. The proof is divided into three steps.

    \textbf{Step 1.} There exists a minimizing sequence $\{u_n\}_{n}\subset\mathcal{N}$ for $\Phi|_{\mathcal{N}}$ which is also a Palais-Smale sequence for $\Phi$. 
    
    It is clear that $v_n:=\frac{u_n}{\|u_n\|}$ is a minimizing sequence for $\Psi$ defined in \eqref{eq:Psi_def}. By Ekeland’s variational principle \cite{Ekeland1974variational}, we may choose $u_n$ so that $\Psi'(v_n)\to 0$. Since $u_n=m(v_n)$, by \cref{prop: PSforPhiAndPsi}, it suffices to verify the assumptions {\rm\refcond{$A_1$}}$\sim${\rm\refcond{$A_3$}} for $E$ and $\Phi$. We have already verified {\rm\refcond{$A_1$}} in \cref{rmk:A1}, and {\rm\refcond{$A_2$}} can be deduced directly from \cref{lem: UniquenessOft_u_G}. Hence, it remains to verify {\rm\refcond{$A_3$}}.
    
    We first show that $\mathcal{N}$ is bounded away from $0$. Indeed, suppose that there exists a sequence $\{u_n\}_n\subset\mathcal{N}$ such that $\|u_n\|\to 0$ as $n\to\infty$, then 
    $$
    \Phi(u_n)\leq \frac{1}{p}\|u_n\|^p\to 0,\ n\to\infty,
    $$
    which contradicts \cref{lem: Phi_BoundedBelow}. It then follows that $s_w=\|s_w w\|$ is bounded away from $0$ uniformly for all $w\in S$, since $s_ww\in\mathcal{N}$.
    
    To verify the second part of {\rm\refcond{$A_3$}}, it suffices to show that for each compact set $W\subset S$ there exists a constant $C_W$ such that $\alpha'_{w}(s)\leq0$ for all $w\in W$ and all $s\in[C_W,+\infty)$. Assume, to the contrary, that there exist sequences $\{w_n\}_n\subset W$ and $s_{n}\to\infty$ such that
    $$
    0<\alpha'_{w_n}(s_n)=\langle\Phi'(s_nw_n),w_n\rangle=s_n^p-\sum_{x\in\mathbb{Z}^N} f(x,s_nw_n)w_n\leq s_n^p-p\sum_{x\in\mathbb{Z}^N} F(x,s_nw_n),
    $$
    where the last inequality follows from \cref{lem: fu>pF}.
    Since $W$ is compact, we may assume that $w_n\to w$ in $S$ for some $w\in W$. Choose $x_0\in\mathbb{Z}^N$ such that $w(x_0)>0$. Then, by \refcond{$S_5$}, for all $M>0$ there exists a constant $n_M\in\mathbb{N}$ such that
    $$
    s_n^p-p\sum_{x\in\mathbb{Z}^N} F(x,s_nw_n)\leq s_n^p-pF(x_0,s_nw_n(x_0))\leq s_n^p-M(s_nw_n(x_0))^{p},\quad\forall n\geq n_M.
    $$
    Since $M>0$ is arbitrary, this is a contradiction.

    \textbf{Step 2.} $b=\inf_{u\in\mathcal{N}}\Phi(u)>0$ is attained by some $\bar{u}\in\mathcal{N}$.
    
    Let $\{u_n\}_{n}\subset\mathcal{N}$ be a minimizing sequence for $\Phi|_{\mathcal{N}}$. By \cref{prop: PhiCoercive}, $\{u_n\}_{n}$ is bounded in $ E$ and hence bounded in $\ell^q(\mathbb{Z}^N)$ for all $q\geq p^*$ by the discrete Sobolev inequality. Passing to a subsequence, we may assume that 
    $$
    \text{$u_n \rightharpoonup \bar{u}$ in $E$,\quad $u_n \rightharpoonup \bar{u}$ in $\ell^q(\mathbb{Z}^N)$ \quad and\quad
    $u_n(x) \to \bar{u}(x)$ for each $x\in\mathbb{Z}^N$. }
    $$
    Since $u_n\in\mathcal{N}$, it follows from conditions \refcond{$S_2$} and \refcond{$S_3$} that for each $\epsilon>0$ there exists a constant $C_\epsilon>0$ such that 
    $$
    \begin{aligned}
        b\leq \Phi(u_n)=\frac{1}{p}\|u_n\|^p-\sum_{x\in\mathbb{Z}^N} F(x,u_n)&=\frac{1}{p}\sum_{x\in\mathbb{Z}^N} f(x,u_n)u_n-\sum_{x\in\mathbb{Z}^N} F(x,u_n)\\
        &\leq \epsilon\|u_n\|_{p^*}^{p^*}+C_\epsilon\|u_n\|_{q}^{q}.
    \end{aligned}
    $$
    By the discrete Sobolev inequality, $\{u_n\}_n$ is also bounded in the space $\ell^{q^*}(\mathbb{Z}^N)$. Choose $\epsilon$ small enough so that $\epsilon\|u_n\|_{p^*}^{p^*}\leq\frac{b}{2}$ for all $n\in \mathbb{N}$. It then follows from the above inequality that $b\leq 2C_\epsilon\|u_n\|_{q}^{q}$. By the Sobolev inequality again, we have
    $$
    0<\frac{b}{2C_{\epsilon}}\leq \|u_n\|_{q}^{q} \leq \|u_n\|_{p^*}^{p^*} \|u_n\|_{\infty}^{q-p^*} \leq S_{p,p^*}^{p^*} \|u_n\|^{p^*} \|u_n\|_{\infty}^{q-p^*},
    $$
    which implies that $\|u_n\|_{\infty}$ is bounded away from $0$. Therefore, there exists $\{x_n\}\subset\mathbb{Z}^N$ and a constant $\delta>0$ such that $u_n(x_n)\geq \delta$ for all $n\in\mathbb{N}$. Let $\Omega:=[0,T)^N\cap\mathbb{Z}^N$. Then there exists $\{k_n\}_n\subset\mathbb{Z}^N$ such that $x_n-k_nT\in\Omega$ for all $n\in\mathbb{N}$. Define the translated sequence by $\bar{u}_n(x):=u_n(x+{k_n}T)$. Since both $V(x)$ and $f(x,u)$ are $T$-periodic, we have $\Phi(\bar{u}_n)=\Phi(u_n)$ for all $n\in\mathbb{N}$. Therefore, replacing $u_n(x)$ with $\bar{u}_n(x)$ if necessary, we may assume that the pointwise limit $\bar{u}\neq 0$.
    
    By \textbf{Step 1}, we may further assume that $\Phi'(u_n)\to 0$ as $n\to\infty$ . Since $u_n(x) \to \bar{u}(x)$ for each $x\in\mathbb{Z}^N$, it follows that for every $v\in C_c(\mathbb{Z}^N)$,
    $$
    \begin{aligned}
    &\langle\Phi'(\bar{u}),v\rangle\\
    =&\sum_{\substack{x,y\in\mathbb{Z}^N\\x\sim y}}\frac{1}{2}|\nabla_{xy}\bar{u}|^{p-2}(\nabla_{xy}\bar{u})(\nabla_{xy}v) +\sum_{x\in\mathbb{Z}^N}V(x)|\bar{u}|^{p-2}\bar{u}v  -\sum_{x\in\mathbb{Z}^N} f(x,\bar{u}){v} \\
    =&\lim_{n\to\infty} \left(\sum_{\substack{x,y\in\mathbb{Z}^N\\x\sim y}}\frac{1}{2}|\nabla_{xy}{u}_n|^{p-2}(\nabla_{xy}{u}_n)(\nabla_{xy}v) +\sum_{x\in\mathbb{Z}^N}V(x)|u_n|^{p-2}u_nv-\sum_{x\in\mathbb{Z}^N} f(x,{u}_n){v}\right)\\
    =&\lim_{n\to\infty}\langle\Phi'(u_n),v\rangle=0,
    \end{aligned}
    $$
    which implies that $\Phi'(\bar{u})=0$, and therefore $\bar{u}\in\mathcal{N}$.
    
    It remains to show that $\Phi(\bar{u})=b$. Since $\bar{u}\neq 0$, by \cref{lem: UniquenessOft_u_G} there exists $\bar{t}>0$ such that $\bar{t}\bar{u}\in\mathcal{N}$. 
    Then, by inequality \eqref{eq:Phi(u)AndPhi(tu)_G} and Fatou's lemma,
    
    \begin{equation*}
    \begin{aligned}
         b=\lim_{n\to\infty}\left(\Phi(u_n)-\frac{1}{p}\langle\Phi'(u_n),u_n\rangle\right)&=\lim_{n\to\infty}\frac{1}{p}\sum_{x\in\mathbb{Z}^N} \left(f(x,u_n)u_n-pF(x,u_n)\right)\\
        &\geq\frac{1}{p}\sum_{x\in\mathbb{Z}^N} \left(f(x,\bar{u})\bar{u}-pF(x,\bar{u})\right)\\
        &=\Phi(\bar{u})-\frac{1}{p}\langle\Phi'(\bar{u}),\bar{u}\rangle\\
        &=\Phi(\bar{u})\geq b.
    \end{aligned}
    \end{equation*}
    This implies that $\Phi(\bar{u})=b$, and $\bar{u}$ is a ground state solution of equation \eqref{eq:MainEquation} on $\mathbb{Z}^N$.
    
    \textbf{Step 3.} Every ground state solution $u$ of equation \eqref{eq:MainEquation} is either positive or negative if $F(x,t)\leq F(x,|t|)$ for all $(x,t)\in\mathbb{Z}^N\times\mathbb{R}$.

    We first show that $|u|$ is also a ground state solution of equation \eqref{eq:MainEquation}. By \cref{lem: UniquenessOft_u_G} there exists $t_0>0$ such that $t_0|u|\in\mathcal{N}$. Since $\Phi(u)=b\leq\Phi(t_0|u|)$, it follows from the expression \eqref{eq: PhiForNehariManifold} that
    \begin{equation}\label{eq: PositiveSolution}
        \begin{aligned}
        \frac{1}{p}\sum_{x\in\mathbb{Z}^N}f(x,u)u-\sum_{x\in\mathbb{Z}^N}F(x,u)=\Phi(u)\leq\Phi(t_0 |u|)&=\frac{t_0^p}{p}\||u|\|^p-\sum_{x\in\mathbb{Z}^N}F(x,t_0|u|)\\
    &\leq \frac{t_0^p}{p}\|u\|^p-\sum_{x\in\mathbb{Z}^N}F(x,t_0u)\\
    &=\frac{t_0^p}{p}\sum_{x\in\mathbb{Z}^N}f(x,u)u-\sum_{x\in\mathbb{Z}^N}F(x,t_0u),
    \end{aligned}
    \end{equation}
    Now fix $x\in\mathbb{Z}^N$ and consider the function 
    $$
    g(x,t):=\frac{t^p}{p}f(x,u(x))u(x)-F(x,tu(x)).
    $$
    By the inequality \eqref{eq: pF_Less_fu}, 
    $$
    g(x,1)-g(x,t)\geq 0,\quad\forall t>0,
    $$
    and the equality is attained if and only if $(1-t)u(x)=0$. It follows that all inequalities in \eqref{eq: PositiveSolution} are, in fact, equalities. Since $u\not\equiv0$, the equality condition implies that $t_0=1$. Hence $|u|=t_0|u|\in\mathcal{N}$ and
    $$
    \begin{aligned}
        \Phi(|u|)=\Phi(t_0|u|)=&\frac{t_0^p}{p}\sum_{x\in\mathbb{Z}^N}f(x,u)u-\sum_{x\in\mathbb{Z}^N}F(x,t_0u)\\
        =&\frac{1}{p}\sum_{x\in\mathbb{Z}^N}f(x,u)u-\sum_{x\in\mathbb{Z}^N}F(x,u)=\Phi(u)=b.
    \end{aligned}
    $$
    This shows that $|u|$ is a ground state solution of equation \eqref{eq:MainEquation}.

%后面再看看
    Next, we show that $|u|$ is strictly positive. Suppose that $\inf_{x\in\mathbb{Z}^N}|u(x)|=0$, then there exists $x_0\in\mathbb{Z}^N$ such that $u(x_0)=0$. Since $|u|$ is a nonnegative ground state solution of equation \eqref{eq:MainEquation}, it follows that
    $$
    0\geq\sum_{y\sim x_0}|\nabla_{x_0y}|u||^{p-2}\nabla_{x_0y}|u|=\Delta_p|u|(x_0)=V(x_0)|u(x_0)|^{p-2}|u(x_0)|-f(x_0,|u(x_0)|)=0.
    $$
    This implies that $u(y)=u(x_0)=0$ for all $y\sim x_0$. As $\mathbb{Z}^N$ is connected, an inductive argument shows that $u(x)\equiv0$, which contradicts the fact that $u$ is a nontrivial solution. 
    
    Now suppose, for contradiction, that $u$ is sign-changing. Define the sets $A:=\{x:u(x)>0\}$ and $B:=\{x:u(x)<0\}$. Then both $A$ and $B$ are nonempty. Since $|u|$ is strictly positive, we also have $\mathbb{Z}^N=A\cup B$. Moreover, as $F(x,u)\leq F(x,|u|)$ and $\||u|\|\leq\|u\|$, it follows from $\Phi(|u|)=\Phi(u)=b$ that $\||u|\|=\|u\|$. This implies that there is no edge in $\mathbb{E}$ connecting the sets $A$ and $B$. However, since both $A$ and $B$ are nonempty and disjoint, the absence of such edges would imply that $\mathbb{Z}^N$ is disconnected, which is a contradiction. Therefore, $u$ must be either strictly positive or strictly negative.
 
\end{proof}

% \begin{remark}
% Actually, \textbf{Step 3} is true for the equation considered in \cite{Hua-XU}. The only use of \refcond{$S_4$} in the proof of \cref{thm: 1} is to show that
% $$
% \frac{t^p}{p}f(x,u)u-F(x,tu)< \frac{1}{p}f(x,u)u-F(x,u),\quad \forall t\in(0,1)\cup(1,\infty).
% $$
% \end{remark}

% When $V(x)\equiv0$ and $f(x,u)=u^{q-1}$ for $q>p^*$, the following result shows that ground state solutions are exactly the minimizers of the discrete Sobolev inequality.

\section{Proof of \cref{thm: 2} and \cref{thm:|u|p}}\label{sec:4} 
\begin{proof}[Proof of \cref{thm: 2}]
Consider the limit equation 
$$
-\Delta_p u + V_{\infty}|u|^{p-2} u= f(u)
$$ 
and its associated energy functional
$$
\Phi_{\infty}(u) := \frac{1}{p} \left(\sum_{\substack{x,y\in\mathbb{Z}^N\\x\sim y}}\frac{1}{2} |\nabla_{xy} u|^p + \sum_{x\in\mathbb{Z}^N}V_{\infty} |u|^p \right) - \sum_{x\in\mathbb{Z}^N} F(u).
$$
By \cref{lem: UniquenessOft_u_G}, the ground state energy of $\Phi$ can be characterized as
$$
b = \inf_{u \in \mathcal{N}} \Phi(u) = \inf_{w \in E\setminus \{0\}} \max_{s > 0} \Phi(sw).
$$
We denote the corresponding quantity for $\Phi_{\infty}$ by
$$
b_{\infty} := \inf_{u \in \mathcal{N}_{\infty}} \Phi_{\infty}(u) = \inf_{w \in E \setminus \{0\}} \max_{s > 0} \Phi_{\infty}(sw).
$$
Since $V_{\infty} = \sup_{x \in \mathbb{Z}^N} V(x)$, it follows that $b_{\infty} \geq b$.
If $V(x) \equiv V_{\infty}$, then this reduces to the case of periodic potentials. Otherwise, $V(x) < V_{\infty}$ on some subset of $\mathbb{Z}^N$.
By \cref{thm: 1}, we know that $b_{\infty}$ is attained at some positive function $u_0 \in \mathcal{N}_{\infty}$, i.e., $\Phi_{\infty}(u_0) = b_{\infty}$.
Then, for any $s > 0$, we have
$$
b_{\infty} = \Phi_{\infty}(u_0) \geq \Phi_{\infty}(su_0) > \Phi(su_0),
$$
which implies that
$$
b_{\infty} > \max_{s > 0} \Phi(su_0) \geq \inf_{u \in E \setminus \{0\}} \max_{s > 0} \Phi(su) =b.
$$

We first show that the Sobolev type inequalities hold for the norm $\|\cdot\|$. Since 
$$
\|V_{-}\|_{\frac{r}{r-p}}< S_{p,r}^p,
$$
by the definition of $S_{p,\infty}$ we have
$$
\begin{aligned}
    \|u\|^p=\|u\|_{\mathcal{D}^{1,p}(\mathbb{Z}^N)}^p-\sum_{x\in\mathbb{Z}^N}V(x)|u(x)|^p
    \geq&\|u\|_{\mathcal{D}^{1,p}(\mathbb{Z}^N)}^p-\|V_{-}\|_{\frac{r}{r-p}}\|u\|_{r}^p\\
    \geq&\left(1-\frac{\|V_{-}\|_{\frac{r}{r-p}}}{S_{p,r}^p}\right)\|u\|_{\mathcal{D}^{1,p}(\mathbb{Z}^N)}^p\geq \epsilon^p\|u\|_{\mathcal{D}^{1,p}(\mathbb{Z}^N)}^p
\end{aligned}
$$
for some sufficiently small $\epsilon>0$. Then by the discrete Sobolev inequality \eqref{eq:DiscreteSobolev}, for all $q\geq p^*$,
$$
\|u\|_q\leq C_{p,q}\|u\|_{\mathcal{D}^{1,p}(\mathbb{Z}^N)}\leq \frac{C_{p,q}}{\epsilon}\|u\|,\quad\forall u\in E.
$$
It follows that all arguments in the proof of \cref{thm: 1} remain valid, except for those that rely on the $T$-periodicity of $V(x)$ and $f(x,u)$.

Recall that in the proof of \cref{thm: 1} the $T$-periodicity is used only to ensure the existence of a minimizing sequence $\{u_n\}_n$ for $\Phi|_{\mathcal{N}}$ with a nonzero weak limit $\bar{u}$. Let $\{u_n\}_n$ be a minimizing sequence for $\Phi|_{\mathcal{N}}$. By arguments analogous to \textbf{Step 2} in the proof of \cref{thm: 1}, we may assume that there exists a sequence $\{x_n\}_n \subset \mathbb{Z}^N$ such that $|u_n(x_n)| \geq \delta > 0$ for all $n\in\mathbb{N}$ and $u_n\rightharpoonup \bar{u}$ in $E$. To show $\bar{u}\neq 0$, it suffices to prove that $\{x_n\}$ is bounded. 

Suppose, by contradiction, that $|x_n| \to \infty$ as $n \to \infty$. We claim that $\Phi_{\infty}'(\bar{u}) = 0$. To verify this, fix $w \in C_c(\mathbb{Z}^N)$ and define $w_n(x) := w(x - x_n)$. By \textbf{Step 1} in the proof of \cref{thm: 1}, we may choose $\{u_n\}_n$ so that
$$
|\langle\Phi'(u_n),w_n\rangle| \leq \|\Phi'(u_n)\|_{E^*} \|w_n\| = \|\Phi'(u_n)\|_{E^*} \|w\| \to 0, \quad \text{as } n \to \infty.
$$
Moreover, define $\bar{u}_n(x):=u_n(x-x_n)$ and assume that $\bar{u}_n \rightharpoonup \bar{u}$ in $E$. It follows that
\begin{align*}
&\langle \Phi'(u_n),w_n\rangle\\ =&\sum_{\substack{x,y\in\mathbb{Z}^N\\x\sim y}}\frac{1}{2}|\nabla_{xy}u_n|^{p-2}(\nabla_{xy}u_n)(\nabla_{xy}w_n)+\sum_{x\in\mathbb{Z}^N}V(x)|u_n|^{p-2}u_nw_n-\sum_{x\in\mathbb{Z}^N}f(u_n)w_n  \\
=&\sum_{\substack{x,y\in\mathbb{Z}^N\\x\sim y}}\frac{1}{2}|\nabla_{xy}\bar{u}_n|^{p-2}(\nabla_{xy}\bar{u}_n)(\nabla_{xy}w)+\sum_{x\in\mathbb{Z}^N}V(x-x_n)|\bar{u}_n|^{p-2}\bar{u}_nw-\sum_{x\in\mathbb{Z}^N}f(\bar{u}_n)w\\
\to& \sum_{\substack{x,y\in\mathbb{Z}^N\\x\sim y}}\frac{1}{2}|\nabla_{xy}\bar{u}|^{p-2}(\nabla_{xy}\bar{u})(\nabla_{xy}w)+\sum_{x\in\mathbb{Z}^N}V_{\infty}|\bar{u}|^{p-2}\bar{u}w-\sum_{x\in\mathbb{Z}^N}f(\bar{u})w \\
=& \langle\Phi_{\infty}'(\bar{u}),w\rangle,\quad n\to\infty,
\end{align*}
which implies that $\langle\Phi_{\infty}'(\bar{u}),w\rangle=0$ for every $w \in C_c(\mathbb{Z}^N)$.
Then, since $f(u)u-pF(u)\geq0$ by \cref{lem: fu>pF}, applying Fatou's lemma yields
\begin{align*}
b + o(1) = \Phi(u_n) - \frac{1}{p} \langle\Phi'(u_n), u_n\rangle &= \sum_{x\in\mathbb{Z}^N} \left( \frac{1}{p} f(u_n) u_n - F(u_n) \right)  \\
&= \sum_{x\in\mathbb{Z}^N} \left( \frac{1}{p} f(\bar{u}_n) \bar{u}_n - F(\bar{u}_n) \right)  \\
&\geq \sum_{x\in\mathbb{Z}^N} \left( \frac{1}{p} f(\bar{u}) \bar{u} - F(\bar{u}) \right)  + o(1) \\
&= \Phi_{\infty}(\bar{u}) - \frac{1}{p} \langle\Phi_{\infty}'(\bar{u}), \bar{u}\rangle + o(1) \\
&= \Phi_{\infty}(\bar{u}) + o(1)\\
&\geq b_{\infty} + o(1),\quad n\to\infty,
\end{align*}
which is a contradiction. Therefore, the sequence $\{x_n\}_n$ must be bounded, and consequently $\bar{u} \neq 0$.

\end{proof}

\begin{proof}[Proof of \cref{thm:|u|p}]
    We first show that $b=(\frac{1}{p}-\frac{1}{q})S_{p,q}^{\frac{pq}{q-p}}$.
    In \cite{Hua-Li}, the authors proved the existence of functions in $\mathcal{D}^{1,p}(\mathbb{Z}^N)$ that attain the best Sobolev constant $S_{p,q}$. Let $v$ be such a function. By \cref{lem: UniquenessOft_u_G}, we have $w:=t_v v\in\mathcal{N}$. It follows that
    $$
    \|w\|=t_v\|v\|=t_vS_{p,q}\|v\|_{q}=S_{p,q}\|w\|_{q},
    $$ 
    and 
    $$
    \|w\|^p-\|w\|_{q}^q=\langle\Phi'(w),w\rangle=0.
    $$
    Hence
    $$
    S_{p,q}=\frac{\|w\|}{\|w\|_q}=\|w\|_q^{\frac{q-p}{p}}=\|w\|^{\frac{q-p}{q}},
    $$
    and therefore,
    \begin{equation}\label{eq: b_LessThan}
         b\leq \Phi(w)=\frac{1}{p}\|w\|^p-\frac{1}{q}\|w\|_q^q=(\frac{1}{p}-\frac{1}{q})S_{p,q}^{\frac{pq}{q-p}}.
    \end{equation}
    On the other hand, by the definition of the best Sobolev constant $S_{p,q}$, we have
    $$
        \Phi(u)=\frac{1}{p}\|u\|^p-\frac{1}{q}\|u\|_q^q\geq \frac{1}{p}\|u\|^p-\frac{1}{qS_{p,q}^q}\|u\|^q, \quad \forall u\in \mathcal{D}^{1,p}(\mathbb{Z}^N).
    $$
    Let $R:=S_{p,q}^{\frac{q}{q-p}}>0$. It follows from \cref{lem: UniquenessOft_u_G} that
    $$
        b = \inf_{u \in \mathcal{N}} \Phi(u) = \inf_{u \in  \mathcal{D}^{1,p}(\mathbb{Z}^N)\setminus \{0\}} \max_{s > 0} \Phi(su)\geq \inf_{u \in  \mathcal{D}^{1,p}(\mathbb{Z}^N),\|u\|=R} \Phi(u) 
        \geq (\frac{1}{p}-\frac{1}{q})S_{p,q}^{\frac{pq}{q-p}}.
    $$
    Combining this with inequality \eqref{eq: b_LessThan}, we conclude that 
    \begin{equation}\label{eq: ValueOf_b}
        b=\Phi(w)=(\frac{1}{p}-\frac{1}{q})S_{p,q}^{\frac{pq}{q-p}}.
    \end{equation}
    
    Since \eqref{eq: b_LessThan} implies that $w=t_vv$ is a ground state solution of equation \eqref{eq:SimpleEq} for each $v$ attaining the best Sobolev constant $S_{p,q}$, it remains to show that every ground state solution of \eqref{eq:SimpleEq} attains $S_{p,q}$. Suppose that $w$ is a ground state solution. Then 
    $$
    \|w\|^p-\|w\|_q^q=\langle\Phi'(w),w\rangle=0\quad \text{and}\quad\Phi(w)=b=(\frac{1}{p}-\frac{1}{q})S_{p,q}^{\frac{pq}{q-p}}.
    $$
    Therefore,
    $$
    (\frac{1}{p}-\frac{1}{q})S_{p,q}^{\frac{pq}{q-p}}=\Phi(w)=\frac{1}{p}\|w\|^p-\frac{1}{q}\|w\|_q^q=(\frac{1}{p}-\frac{1}{q})\|w\|^p.
    $$
    It follows that
    $$
    \|w\|=S_{p,q}^{\frac{q}{q-p}}=S_{p,q} (S_{p,q}^{\frac{q}{q-p}})^{\frac{p}{q}}=S_{p,q}\|w\|^{\frac{p}{q}}=S_{p,q}\|w\|_q,
    $$
    which completes the proof.
\end{proof}

\section{Proof of \cref{thm:InfiniteSolution}} \label{sec:5}
The argument in this section adopts the approach developed in the proof of \cite[Theorem 1.2]{Szulkin-Weth-paper}. We first introduce the definition and some basic properties of Krasnoselskii genus, see \cite{StruweVariationalMethods} for details.

\begin{definition}
Let $E$ be a Banach space. For each closed and symmetric nonempty subsets $A \subset E$, i.e., $A = -A = \overline{A} \neq \emptyset$, define the Krasnoselskii genus of $A$ as
$$
\gamma(A) = 
\begin{cases}
\inf\{m \in \mathbb{N}^+ : \exists h \in C^0(A; \mathbb{R}^m \setminus \{0\}), h(-u) = -h(u)\}, \\
\infty, \text{ if } \{\cdots\} = \emptyset, \text{ in particular, if } 0 \in A,
\end{cases}
$$
and define $\gamma(\emptyset) = 0$.
\end{definition}

\begin{proposition}\label{prop: GenusProperties}
Let $A, A_1, A_2$ be closed and symmetric subsets of a Banach space $E$. Then
\begin{enumerate}[label=\rm(\roman*)]
    \item $\gamma(A) \geq 0$ and $\gamma(A) = 0$ if and only if $A = \emptyset$,
    \item $\gamma(A_1 \cup A_2) \leq \gamma(A_1) + \gamma(A_2)$,
    \item if $f : A \to f(A)$ is odd and continuous, then $\gamma(A) \leq \gamma(f(A))$,
    \item if $A$ is compact, $0 \notin A$ and $A \neq \emptyset$, then there exist open set $U \supset A$ such that $\gamma(U) = \gamma(A)$.
\end{enumerate}
\end{proposition}

We now recite some notations introduced in \cite{Szulkin-Weth-paper}. For $d \geq e \geq c$, set
\begin{alignat*}{3}
\Phi^d &:= \{u \in \mathcal{N} : \Phi(u) \leq d\}, &\quad 
\Phi_e &:= \{u \in \mathcal{N} : \Phi(u) \geq e\}, &\quad 
\Phi_e^d &:= \Phi^d \cap \Phi_e, \\
\Psi^d &:= \{w \in S : \Psi(w) \leq d\}, &\quad 
\Psi_e &:= \{w \in S : \Psi(w) \geq e\}, &\quad 
\Psi_e^d &:= \Psi^d \cap \Psi_e, \\
K &:= \{w \in S : \Psi'(w) = 0\}, &\quad 
K_d &:= \{w \in {K} : \Psi(w) = d\}, &\quad 
\end{alignat*}

Since $f(x,u)$ is odd in $u$, it follows that $K$ is symmetric with respect to the origin, i.e., $u \in K$ if and only if $-u \in K$. Moreover, as $K$ is invariant under the action of $\mathbb{Z}^N$ and $\mathcal{O}(u)\cap\mathcal{O}(-u)=\varnothing$, we can select a subset $\mathcal{F}$ of $K$ such that $\mathcal{F}=-\mathcal{F}$ and $|\mathcal{F}\cap\mathcal{O}(u)|=1$ for each $u\in K$. Consequently, to prove \cref{thm:InfiniteSolution}, it suffices to show that $\mathcal{F}$ is infinite. We proceed by contradiction and assume that
\begin{equation}\label{eq: FinteSetAssumption}
    \mathcal{F}\text{ is a finite set}.
\end{equation}
The following lemmas are derived under this assumption.

\begin{lemma}\label{lem: DiscretenessOfSolutions}
$\kappa := \inf\{\|v - w\|: v, w \in K, v \neq w\} > 0$.
\end{lemma}

\begin{proof}
Since $\mathcal{F}$ is finite, we can write $\mathcal{F}=\{u_i\}_{i=1}^m$. Then, by the definition of $\mathcal{F}$,
$$
\begin{aligned}
    \kappa &=\inf\{\|u_i(\cdot-kT) - u_j(\cdot-lT)\|:  1\leq i,j\leq m,\ k,l\in\mathbb{Z}^N,\ u_i(\cdot-kT) \neq u_j(\cdot-lT)\}\\
    &=\inf\{\|u_i(\cdot-kT) - u_j\|:  1\leq i,j\leq m,\ k\in\mathbb{Z}^N,\ u_i(\cdot-kT) \neq u_j\}.
\end{aligned}
$$
Suppose, by contradiction, that $\kappa=0$. Then there exist sequences $\{v_n\}_n,\{w_n\}_n\subset\mathcal{F}$ and $\{k_n\}\subset\mathbb{Z}^N$ such that $v_n(\cdot-k_nT)\neq w_n$ for all $n\in\mathbb{N}$ and 
$$
\lim_{n\to\infty}\|v_n(\cdot-k_nT)- w_n\|=0.
$$
Passing to a subsequence, we may assume that $v_n= v\in\mathcal{F}$ and $w_n= w\in\mathcal{F}$ for all $n\in\mathbb{N}$. If $\{k_n\}_n$ is bounded in $\mathbb{Z}^N$, then there exists a subsequence $\{k_{n_l}\}_l$ and $k\in\mathbb{Z}^N$ such that $k_{n_l}=k$ for all $l\in\mathbb{N}$. It then follows that
$$
\|v(\cdot - kT) - w\| =\|v_{n_l}(\cdot - k_{n_l} T) - w_{n_l}\| =\lim_{l \to \infty} \|v_{n_l}(\cdot - k_{n_l} T) - w_{n_l}\|= 0,
$$
which is a contradiction. If $\{k_n\}_n$ is unbounded in $\mathbb{Z}^N$, then since the sequence $\{v(\cdot-k_nT)\}_n$ is bounded in $E$, passing to a subsequence, we may assume that $v(\cdot-k_nT)\rightharpoonup v_0\in E$. However, since $v(\cdot-k_nT)\to0$ pointwise as $n\to\infty$, it follows that $v_0=0$. Hence,
$$
\lim_{n\to\infty}\|v(\cdot-k_{n}T)-w\|\geq\|w\|=1,
$$ 
which is again a contradiction. Therefore, we conclude that $\kappa > 0$.
\end{proof}

The following lemma plays a crucial role in the proof and will be established in the Appendix.
\begin{lemma}[Discreteness of PS-sequences]\label{lem: DiscretenessOfPS-sequences}
Suppose that $d \geq b$, where $$
b=\inf_{u\in\mathcal{N}}\Phi(u)=\inf_{v\in S}\Psi(v).
$$ 
If $\{u_n\}_n, \{v_n\}_n \subset \Psi^d$ are two Palais-Smale sequences for $\Psi$, then either $\|u_n - v_n\| \to 0$ as $n \to \infty$ or $\varlimsup_{n\to\infty} \|u_n - v_n\| \geq \rho(d) > 0$, where $\rho(d)$ is a constant that only depends on $d$.
\end{lemma}

It is well known (see e.g. \cite[Lemma II.3.9]{StruweVariationalMethods}) that $\Psi$ admits a pseudo-gradient vector field, i.e., there exists a Lipschitz continuous map $H : S \setminus K \to TS$ such that $H(w) \in T_w S$ for all $w \in S \setminus K$ and
\begin{equation}
\begin{cases}
\|H(w)\| < 2 \|\Psi'(w)\| \\
\langle \Psi'(w), H(w) \rangle > \frac{1}{2} \|\Psi'(w)\|^2
\end{cases} \quad \text{for all } w \in S \setminus K,
\end{equation}
where
$$
\|\Psi'(w)\| := \sup \{ \langle \Psi'(w), v \rangle : v \in T_w S, \|v\| \leq 1 \}.
$$ 
Since $\Psi$ is an even functional, i.e., $\Psi(w)=\Psi(-w)$ for all $w\in S$, it follows that $\Psi'(-w) = -\Psi'(w)$. Therefore, by replacing $H(w)$ with $(H(w) - H(-w))/2$ if necessary, we may assume that $H$ is an odd vector field, i.e., $H(-w) = -H(w)$ for all $w \in S$.

We denote by $\eta : \mathcal{G} \to S \setminus K$ the ($\Psi$-decreasing) flow generated by the pseudo-gradient vector field $-H$, that is, the solution to the initial value problem
\begin{equation}\label{eq: PsiDecreasing_Flow}
\begin{cases}
\frac{d}{dt} \eta(t,w) = -H(\eta(t,w)), \\
\eta(0,w) = w.
\end{cases}
\end{equation}

Here, the domain $\mathcal{G} \subset \mathbb{R} \times (S \setminus K)$ is given by
$$
\mathcal{G} := \{ (t,w) : w \in S \setminus K,\ T^-(w) < t < T^+(w) \},
$$
where $T^-(w)$ and $T^+(w)$ denote the maximal negative and positive existence times of the trajectory through $w$ respectively. The following lemmas concern the limiting behavior of the flow $\eta(t,w)$ as $t$ approaches the maximal positive existence time. For the sake of completeness, we provide the proofs below, which follow the same arguments as in \cite{Szulkin-Weth-paper}.

\begin{lemma}\label{lem: DeformLemma1}
For every $w \in S$ the limit $\lim_{t\to T^+(w)} \eta(t,w)$ exists and is a critical point of $\Psi$.
\end{lemma}

\begin{proof}
Fix $w\in S$ and set $d:= \Psi(w)$. Since $\Psi(\eta(t,w))$ is decreasing and bounded below by $b=\inf_{v\in S}\Psi(v)$, it converges to some $c\geq b$ as $t\to T^+(w)$. We distinguish two cases.

\textbf{Case 1.} $T^+(w)<\infty$. For all $0\leq t_1< t_2<T^+(w)$, we have
$$
\begin{aligned}
    \|\eta(t_2,w)-\eta(t_1,w)\|^2 \leq\left(\int_{t_1}^{t_2}\|H(\eta(t,w))\| dt\right)^2
    &\leq (t_2-t_1)\int_{t_1}^{t_2}\|H(\eta(t,w))\|^2 dt\\
    &\leq 8(t_2-t_1)\int_{t_1}^{t_2}\langle\Psi'(\eta(t,w)),H(\eta(t,w))\rangle dt\\
    &=8(t_2-t_1)(\Psi(\eta(t_1,w))-\Psi(\eta(t_2,w)))\\
    &\leq 8T^+(w)(\Psi(\eta(t_1,w))-\Psi(\eta(t_2,w))),
\end{aligned}
$$
which implies that $\lim_{t\to T^+(w)}\eta(t,w)$ exists. This limit must be a critical point of $\Psi$, for otherwise the trajectory $t\to\eta(t,w)$ could be extended beyond $T^+(w)$.

\textbf{Case 2.} $T^+(w)=\infty$. We first show that $\lim_{t\to T^+(w)}\eta(t,w)$ exists. Suppose, by contradiction, that this limit does not exist. Then there exists $0<\epsilon_0<\rho(d)/5$ and a sequence $t_n\to\infty$ such that $\|\eta(t_{n+1},w)-\eta(t_n,w)\|=\epsilon_0$. Define
$$
t_n^1:=\inf\{t\in[t_n,t_{n+1}]:\|\eta(t,w)-\eta(t_n,w)\|=\epsilon_0\},\ \kappa_n:=\inf\{\|\Psi'(\eta(t,w))\|:t\in[t_n,t_n^1]\}.
$$
Then
$$
\begin{aligned}
    \epsilon_0=\|\eta(t_n^1,w)-\eta(t_n,w)\|\leq\int_{t_n}^{t_n^1}\|H(\eta(t,w))\| dt&\leq2\int_{t_n}^{t_n^1}\|\Psi'(\eta(t,w)\| dt\\
    &\leq \frac{2}{\kappa_n}\int_{t_n}^{t_n^1}\|\Psi'(\eta(t,w)\|^2 dt\\
    &\leq \frac{4}{\kappa_n}\int_{t_n}^{t_n^1}\langle\Psi'(\eta(t,w)),H(\eta(t,w)) \rangle dt\\
    &\leq \frac{4}{\kappa_n}(\Psi(\eta(t_n,w))-\Psi(\eta(t_n^1,w))).
\end{aligned}
$$
Since $\Psi(\eta(t_n,w))-\Psi(\eta(t_n^1,w))\to0$ as $n\to\infty$, it follows that $\kappa_n\to0$. Hence, there exists $\tau_n^1\in[t_n,t_n^1]$ such that $\|\Psi'(\eta(\tau_n^1,w))\|\to0$. Similarly, define
$$
t_n^2 :=\sup\{t\in[t_n,t_{n+1}]:\|\eta(t_{n+1},w)-\eta(t,w)\|=\epsilon_0\},
$$
there exists $\tau_n^2\in[t_n^2,t_{n+1}]$ such that $\|\Psi'(\eta(\tau_n^2,w))\|\to0$. Since $t\mapsto\Psi(\eta(t,w))$ is decreasing, both $\{\eta(\tau_n^1,w)\}_n$ and $\{\eta(\tau_n^2,w)\}_n$ are Palais-Smale sequences for $\Psi$. Moreover,
$$
\begin{aligned}
    \|\eta(\tau_n^1,w)-\eta(\tau_n^2,w)\|\geq&\|\eta(t_{n+1},w)-\eta(t_n,w)\|-\|\eta(t_{n+1},w)-\eta(\tau_n^2,w)\|\\
    &-\|\eta(t_n,w)-\eta(\tau_n^1,w)\|\\
    \geq& 3\epsilon_0-\epsilon_0-\epsilon_0=\epsilon_0,
\end{aligned}
$$
and
$$
\begin{aligned}
    \|\eta(\tau_n^1,w)-\eta(\tau_n^2,w)\|\leq&\|\eta(t_n,w)-\eta(\tau_n^1,w)\|+\|\eta(t_{n+1},w)-\eta(\tau_n^2,w)\|\\
    &+\|\eta(t_{n+1},w)-\eta(t_n,w)\|\\
    \leq& \epsilon_0+\epsilon_0+3\epsilon_0=5\epsilon_0<\rho(d),
\end{aligned}
$$
which contradicts \cref{lem: DiscretenessOfPS-sequences}. Therefore, the limit $\hat{w}:=\lim_{t\to T^+(w)}\eta(t,w)$ exists. 
It remains to show that $\hat{w}$ is a critical point of $\Psi$. If not, then since $\Psi\in C^1(S,\mathbb{R})$, we would have
$$
\varliminf_{t\to\infty}\frac{d}{dt}\Psi(\eta(t,w))=\varliminf_{t\to\infty} \langle\Psi'(\eta(t,w)),H(\eta(t,w)) \rangle\geq \frac{1}{2}\varliminf_{t\to\infty}\|\Psi'(\eta(t,w))\|^2=\frac{1}{2}\|\Psi'(\hat{w})\|^2>0,
$$
which is impossible since $\Psi(\eta(t, w))$ is decreasing and convergent. Hence, $\hat{w}$ must be a critical point of $\Psi$. The proof is complete.
\end{proof}

For a subset $P \subset S$ and $\delta > 0$, we define the $\delta$-neighborhood of $P$ by
$$
U_\delta(P) := \{w \in S: \text{dist}(w,P) < \delta\}. 
$$

\begin{lemma}\label{lem: DeformLemma2}
Let $d > b:=\inf_{w\in S}\Psi(w)$. Then for every $\delta > 0$ there exists $\epsilon = \epsilon(\delta) > 0$ such that
\begin{enumerate}
\item[\rm(a)] $\Psi^{d+\epsilon}_{d-\epsilon} \cap K = K_d$ and
\item[\rm(b)] $\lim_{t\to T^+(w)} \Psi(\eta(t,w)) < d - \epsilon$ for $w \in \Psi^{d+\epsilon} \setminus U_\delta(K_d)$.
\end{enumerate}
\end{lemma}
\begin{proof}
     By assumption \eqref{eq: FinteSetAssumption}, statement {\rm(a)} holds for all sufficiently small $\epsilon > 0$. Since $\mathcal{F}$ is finite and $\Psi'(w)$ is invariant under the action of $\mathbb{Z}^N$, there exists $\delta_0>0$ such that for all $\delta<\delta_0$, 
     $$
     A_{\delta}:=\sup\{\|\Psi'(w)\|:{w\in U_{\delta}(K_d)} \}<\infty.
     $$
    Without loss of generality, we may assume that $U_\delta(K_d)\subset\Psi^{d+1}$ and  $\delta<\min\{\delta_0,\rho(d+1)\}$.
    Define    
     $$
        \tau:=\inf\{\|\Psi'(w)\|:w\in U_{\delta}(K_d)\setminus U_{\frac{\delta}{2}}(K_d)\}.
     $$
     We first show that $\tau>0$. Suppose, by contradiction, that there exists a sequence $\{w_n\}_n\subset U_{\delta}(K_d)\setminus U_{\frac{\delta}{2}}(K_d)$ such that $\|\Psi'(w_n)\|\to0$. By the finiteness of $\mathcal{F}$ and the $\mathbb{Z}^N$-invariance of $\Psi'(w)$, we may assume that $\{w_n\}_n\subset U_{\delta}(w_0)\setminus U_{\frac{\delta}{2}}(w_0)$ for some $w_0\in K_d$. Let $\{v_n\}_n\subset U_{\delta}(w_0)$ be a sequence in $S$ that converges to $w_0$. Then $\|\Psi'(v_n)\|\to0$ and 
     $$
        \frac{\delta}{2}\leq\varlimsup_{n\to\infty}\|w_n-v_n\|\leq\delta<\rho(d+1),
     $$
     which contradicts \cref{lem: DiscretenessOfPS-sequences}. Hence $\tau>0$. 

     Next, we proceed to prove {\rm(b)}. By \cref{lem: DeformLemma1}, the limit $\hat{w}:=\lim_{t\to T^+(w)}\eta(t,w)$ exists and is a critical point of $\Psi$ for every $w\in S$. Therefore, it suffices to show that $\hat{w}\notin K_d$ for all $w\in \Psi^{d+\epsilon} \setminus U_\delta(K_d)$. Suppose, by contradiction, that there exists $w\in \Psi^{d+\epsilon} \setminus U_\delta(K_d)$ such that $\hat{w}\in K_d$. One can choose an interval $[t_1,t_2]\subset[0,+\infty)$ such that 
     $$
     \|\eta(t_1,w)-\hat{w}\|=\delta,\ \|\eta(t_2,w)-\hat{w}\|=\frac{\delta}{2}, \text{ and } \eta(t,w)\in U_{\delta}(\hat{w})\setminus U_{\frac{\delta}{2}}(\hat{w}) \text{ for all } t\in(t_1,t_2).
     $$ 
     Then
     $$
        \frac{\delta}{2}\leq\|\eta(t_2,w)-\eta(t_1,w)\|\leq\int_{t_1}^{t_2}\|H(\eta(t,w))\| dt\leq 2\int_{t_1}^{t_2}\|\Psi'(\eta(t,w))\| dt\leq 2A_{\delta}(t_2-t_1),
     $$
     which implies that $t_2-t_1\geq\frac{\delta}{4A_{\delta}}$. Furthermore,
     $$
     \begin{aligned}
         \Psi(\eta(t_2,w))-\Psi(\eta(t_1,w))&=-\int_{t_1}^{t_2}\langle\Psi'(\eta(t,w)),H(\eta(t,w))\rangle dt\\
         &\leq-\frac{1}{2} \int_{t_1}^{t_2}\|\Psi'(\eta(t,w))\|^2 dt\leq -\frac{\tau^2(t_2-t_1)}{2}\leq-\frac{\tau^2\delta}{8A_{\delta}}.
     \end{aligned}
     $$
     Now choose $\epsilon<\frac{\tau^2\delta}{8A_{\delta}}$ small enough so that {\rm (a)} holds. It follows that
     $$
     \Psi(\eta(t_2,w))\leq \Psi(\eta(t_1,w))-\frac{\tau^2\delta}{8A_{\delta}}\leq d+\epsilon-\frac{\tau^2\delta}{8A_{\delta}}<d,
     $$
     which contradicts the assumption that $\eta(t,w)\to\hat{w}\in K_d$.
\end{proof}

\begin{proof}[Proof of \cref{thm:InfiniteSolution}]
For each $k\in\mathbb{N}$, define
$$
c_k:=\inf\{d\in\mathbb{R}:\gamma(\Psi^d)\geq k\}.
$$
It suffices to prove the following stronger statement:
\begin{equation}\label{eq: StrongClaimForInfiniteSolutions}
    K_{c_k}\neq\varnothing\ \text{ and } \ c_k<c_{k+1} \ \text{ for all }k\in\mathbb{N}.
\end{equation}

Fix $k\in\mathbb{N}$. Since $K$ is a discrete set by \cref{lem: DiscretenessOfSolutions}, we have $\gamma(K_{c_k})=0$ or $1$. By \cref{prop: GenusProperties}, there exists $0<\delta<\frac{\kappa}{2}$ such that $\gamma(\overline{U})=\gamma(K_{c_k})$, where $U:=U_{\delta}(K_{c_k})$ is the $\delta$-neighborhood of $K_{c_k}$ and $\kappa$ is the constant given in \cref{lem: DiscretenessOfSolutions}. Choose $\epsilon=\epsilon(\delta)>0$ small enough so that the conclusion in \cref{lem: DeformLemma2} holds. It follows that for all $w\in \Psi^{c_k+\epsilon}\setminus U$ there exists $t\in[0,T^+(w))$ such that $\Psi(\eta(t,w)) < c_k- \epsilon$. Hence, we may define the entrance time map by
\begin{equation}
    e:\Psi^{c_k+\epsilon}\setminus U\to[0,\infty), \quad e(w):=\inf\{t\in[0,T^+(w)):\Psi(\eta(t,w))= c_k- \epsilon\}.
\end{equation}
Note that $e(w)$ is the unique $t\in[0,T^+(w))$ satisfying $\Psi(\eta(t,w))= c_k- \epsilon$ as $\Psi(\eta(t,w))$ is strictly decreasing for $w\in\Psi^{c_k+\epsilon}\setminus U$. Since $c_k-\epsilon$ is not a critical value for $\Psi$,
$$
\frac{d}{dt} \Psi(\eta(t,w)) = -\langle \Psi'(\eta(t,w)),H(\eta(t,w))\rangle<-\frac{1}{2}\|\Psi'(\eta(t,w))\|^2<0.
$$
By a general implicit function theorem, it follows that the map $e$ is continuous. Moreover, since $H$ is an odd vector field, $e$ is also an even map. We thus obtain an odd and continuous map
$$
h:\Psi^{c_k+\epsilon}\setminus U\to\Psi^{c_k-\epsilon},\quad h(w)=\eta(e(w),w).
$$
It then follows from \cref{prop: GenusProperties} that $\gamma({\Psi^{c_k+\epsilon}\setminus U})\leq\gamma(\Psi^{c_k-\epsilon})\leq k-1$. Therefore,
$$
k\leq \gamma(\Psi^{c_k+\epsilon})\leq \gamma(\overline{U})+\gamma({\Psi^{c_k+\epsilon}\setminus U})\leq \gamma(\overline{U})+k-1=\gamma(K_{c_k})+k-1.
$$
This inequality implies that $\gamma(K_{c_k}) \geq 1$ and $\gamma(\Psi^{c_k+\epsilon})\leq k$. Since $\gamma(K_{c_k}) \in \{0, 1\}$ by previous argument, we conclude that $\gamma(K_{c_k}) = 1$, and hence $K_{c_k}\neq\varnothing$. Moreover, $\gamma(\Psi^{c_k+\epsilon})\leq k$ implies that $c_{k+1} \geq c_k + \epsilon > c_k$, which completes the proof of \eqref{eq: StrongClaimForInfiniteSolutions}.

By \eqref{eq: StrongClaimForInfiniteSolutions}, there exits an infinite sequence $\{w_k\}_k$ of geometrically distinct critical points of $\Psi$ with $\Psi(w_k)=c_k$, contradicting the assumption \eqref{eq: FinteSetAssumption}. The proof is complete.

\end{proof}

% $e:\Psi^{c_k+\epsilon}\setminus U\to[0,\infty)$ by letting $e(w)$ to be the unique $t\in[0,T^+(w))$ satisfying $\Psi(\eta(t,w))= c_k- \epsilon$.

\begin{proof}[Proof of \cref{thm:Cayley}]
    Let $\Gamma/H = \{\rho_1, \ldots, \rho_m\}$ be the set of finitely many orbits of the action of $H$ on $\Gamma$. Choose a fundamental domain $\Omega = \{z_1, \ldots, z_m\}\subset\Gamma$ such that $z_i \in \rho_i$, $1 \leq i \leq m$. The volume growth $V(r)\geq Cr^D$ ensures the validity of the Sobolev inequality \eqref{eq: SobolevOnCayley} on $\Gamma$. It follows that all arguments in the proof of \cref{thm: 1}, \cref{thm: 2} and \cref{thm:InfiniteSolution} remain valid, except for those relying on translational invariance to obtain a minimizing sequence with nonzero pointwise limit $\bar{u}$. 
    
    Now suppose we have established that $|u_n(x_n)| \geq \delta > 0$. For each $x_n$, there exists $h_n \in H$ such that $h_n x_n \in \Omega$. Define $\bar{u}_n := u_n \circ h_n^{-1}$. By the $H$-invariance of $V$ and $f$, the sequence $\{\bar{u}_n\}_n$ is again a minimizing sequence. Moreover,
    $$
    |\bar{u}_n(h_nx_n)|=|u_n(x_n)|\geq \delta>0,
    $$
    which implies that $\|\bar{u}_n\|_{\ell^\infty(\Omega)} \geq \delta$. Since $\Omega$ is a finite set, it follows that the pointwise limit of $\{\bar{u}_n\}$ is nonzero. The remainder of the proof then follows the same reasoning as in the previous theorems.
\end{proof}

\appendix

\section*{Appendix}
\phantomsection  % 让定理编号与章节关联
\setcounter{section}{1}% 手动设置章节计数器值
\setcounter{theorem}{0}%       % 重置定理计数器

\begin{lemma}\label{lem: pInequality}
    Assume that $a_1,a_2\in\mathbb{R}$ and $p\geq 2$. Then
    $$
        |a_1-a_2|^{p}\leq 2^{p-1}(|a_1|^{p-2}a_1-|a_2|^{p-2}a_2)(a_1-a_2).
    $$
\end{lemma}
\begin{proof}
    We first consider the case where $a_1\cdot a_2\leq0$. Then
    $$
    \begin{aligned}
        |a_1-a_2|^{p}=(|a_1|+|a_2|)^p&\leq(2\max\{|a_1|,|a_2|\})^{p-1}(|a_1|+|a_2|)\\
        &\leq 2^{p-1}(|a_1|^{p-1}+|a_2|^{p-1})(|a_1|+|a_2|)\\
        &=2^{p-1}(|a_1|^{p-2}a_1-|a_2|^{p-2}a_2)(a_1-a_2).
    \end{aligned}
    $$
    
    Now suppose that $a_1\cdot a_2>0$. Without loss of generality, we may assume that $a_1> a_2>0$. After dividing both sides by $a_2^{p-1}(a_1-a_2)$, it suffices to prove the inequality
    $$
     (t-1)^{p-1}\leq2^{p-1}(t^{p-1}-1) \quad\forall t>1,
    $$
    which can be verified directly using differentiation.
\end{proof}

\begin{lemma}\label{lem: LimitOfPS_Remians}
    Let $d\geq b$, where $b=\inf_{v\in S}\Phi(v)$. Suppose that 
    $$
    \{w_n\}_n\subset\Phi^d = \{u \in \mathcal{N} : \Phi(u) \leq d\}
    $$ 
    is a Palais-Smale sequence for $\Phi$ and $w_n\rightharpoonup w\in E\setminus\{0\}$, then $w\in\Phi^d\cap\Lambda$, where $\Lambda:=\{u\in E:\Phi'(u)=0\}.$
\end{lemma}
\begin{proof}
     Since $w_n\to w$ pointwise as $n\to\infty$, for all $v\in C_c(\mathbb{Z}^N)$,
     $$
        |\langle\Phi'(w),v\rangle|=\lim_{n\to\infty} |\langle\Phi'(w_n),v\rangle|\leq \varliminf_{n\to\infty}\|\Phi'(w_n)\|_{E^*}\|v\| = 0.
     $$
     Hence $w\in\Lambda$. It remains to show that $\Phi(w)\leq d$. By identities \eqref{eq: ExpressionOfPhi} and \eqref{eq: ExpressionOf<Phi'u,u>}, we have 
     $$
     \Phi(u)=\frac{1}{p}\langle\Phi'(u),u\rangle+\sum_{x\in\mathbb{Z}^N}\left(\frac{f(x,u)u}{p}-F(x,u)\right)
     $$
     for all $u\in E$. Since $f(x,u)u\geq p F(x,u)$ by \cref{lem: fu>pF} and $w,w_n\in\mathcal{N}$, it follows from Fatou's lemma that
     $$
     \begin{aligned}
         \Phi(w)=\sum_{x\in\mathbb{Z}^N}\left(\frac{f(x,w)w}{p}-F(x,w)\right)&\leq\varliminf_{n\to\infty}\sum_{x\in\mathbb{Z}^N}\left(\frac{f(x,w_n)w_n}{p}-F(x,w_n)\right)\\
         &=\varliminf_{n\to\infty} \Phi(w_n)\leq d,
     \end{aligned}
     $$
     This completes the proof.
\end{proof}

\begin{proof}[Proof of \cref{lem: DiscretenessOfPS-sequences}]
     Let $\hat{u}_n := {m}(u_n)$ and $\hat{v}_n := {m}(v_n)$ for all $n \in \mathbb{N}$. Then by \cref{prop: PhiCoercive}, $\{\hat{u}_n\}_n, \{\hat{v}_n\}_n \subset \Phi^d$ are two Palais-Smale sequences for $\Phi$ with norms bounded by 
    $$
        \nu_d := \sup_{u \in \Phi^d} \|u\| < \infty
    $$
    for all $d \ge b$. Fix $q > p^*$ and consider the following two cases.
        
    \textbf{Case 1.} $\|\hat{u}_n-\hat{v}_n\|_q\to 0$ as $n\to\infty$. By \cref{lem: pInequality} we have 
    $$
    \sum_{\substack{x,y\in\mathbb{Z}^N\\x\sim y}} |\nabla_{xy} (\hat{u}_n - \hat{v}_n)|^p\le 2^{p-1} \sum_{\substack{x,y\in\mathbb{Z}^N\\x\sim y}} \big(|\nabla_{xy} \hat{u}_n|^{p-2} \nabla_{xy}\hat{u}_n - |\nabla_{xy} \hat{v}_n|^{p-2} \nabla_{xy}\hat{v}_n\big) (\nabla_{xy} (\hat{u}_n - \hat{v}_n)) 
    $$
    and
    $$
    \sum_{x\in\mathbb{Z}^N} V(x)|\hat{u}_n-\hat{v}_n|^p\leq2^{p-1} \sum_{x\in\mathbb{Z}^N}V(x) (|\hat{u}_n|^{p-2}\hat{u}_n-|\hat{v}_n|^{p-2}\hat{v}_n)(\hat{u}_n-\hat{v}_n).
    $$
    It then follows from the expression of $\langle \Phi'(u),v\rangle$ in \eqref{eq: DerivativeOfPhi} and the conditions \refcond{$S_2$}$\sim$\refcond{$S_3$} that for each fixed $\epsilon>0$ there exist $n_{\epsilon}$ and $C_{\epsilon}>0$ such that for all $n\geq n_{\epsilon}$,
    $$
    \begin{aligned}
        \|\hat{u}_n-\hat{v}_n\|^p&\leq2^{p-1}\left[\langle \Phi'(\hat{u}_n)-\Phi'(\hat{v}_n),\hat{u}_n-\hat{v}_n\rangle+\sum_{x\in\mathbb{Z}^N}(f(x,\hat{u}_n)-f(x,\hat{v}_n))(\hat{u}_n-\hat{v}_n)\right]\\
        &\leq \epsilon\|\hat{u}_n-\hat{v}_n\|+\sum_{x\in\mathbb{Z}^N}\left[\epsilon (|\hat{u}_n|^{p^*-1}+|\hat{v}_n|^{p^*-1})+C_\epsilon(|\hat{u}_n|^{q-1}+|\hat{v}_n|^{q-1})\right]|\hat{u}_n-\hat{v}_n|\\
        &\leq \epsilon\|\hat{u}_n-\hat{v}_n\|+4\epsilon\max\{\|\hat{u}_n\|_{p^*}^{p^*},\|\hat{v}_n\|_{p^*}^{p^*}\}+C_{\epsilon}(\|\hat{u}_n\|_{q}^{q-1}+\|\hat{v}_n\|_{q}^{q-1})\|\hat{u}_n-\hat{v}_n\|_{q}\\
        &\leq \epsilon [\|\hat{u}_n-\hat{v}_n\|+C_1]+C_{\epsilon}C_2\|\hat{u}_n-\hat{v}_n\|_{q},
    \end{aligned}
    $$
    where the third inequality follows from H$\rm\ddot{o}$lder's inequality and $C_1,C_2>0$ are constants independent of $\epsilon$. 
    Therefore, we have
    $$
    \varlimsup_{n\to\infty}\|\hat{u}_n-\hat{v}_n\|^p\leq\epsilon(\varlimsup_{n\to\infty}\|\hat{u}_n-\hat{v}_n\|+C_1).
    $$
    Since $\epsilon$ is arbitrary, it follows that $\lim_{n\to\infty}\|\hat{u}_n-\hat{v}_n\|=0$. Then by \cref{prop: mIsHomeomorphism} we conclude that $\lim_{n\to\infty}\|{u}_n-{v}_n\|=0$ as \refcond{$A_2$} and \refcond{$A_3$} are verified in the proof of \cref{thm: 1}.
    %关于条件验证可以单独写一个引理

    \textbf{Case 2.} $\delta:=\varlimsup\limits_{n\to\infty}\|\hat{u}_n-\hat{v}_n\|_q>0$. Since 
    $$
    \sup_n\|\hat{u}_n-\hat{v}_n\|\leq\sup_n\|\hat{u}_n\|+\sup_n\|\hat{v}_n\|\leq 2\nu_d,
    $$
    it follows from the discrete Sobolev inequality that
    $$
        \delta^q=\varlimsup_{n\to\infty}\|\hat{u}_n-\hat{v}_n\|_q^q\leq\varlimsup_{n\to\infty}\|\hat{u}_n-\hat{v}_n\|_{\infty}^{q-p^*}\|\hat{u}_n-\hat{v}_n\|_{p^*}^{p^*}\leq (2S_{p,p^*}\nu_d)^{p^*}\varlimsup_{n\to\infty}\|\hat{u}_n-\hat{v}_n\|_{\infty}^{q-p^*},
    $$
    which implies that $\lambda:=\varlimsup\limits_{n\to\infty}\|\hat{u}_n-\hat{v}_n\|_{\infty}>0$. Passing to a subsequence, we may assume that for all $n\in \mathbb{N}$ there exists $x_n\in\mathbb{Z}^N$ such that $|\hat{u}_n(x_n)-\hat{v}_n(x_n)|\geq\lambda/2$. Since both $\Phi$ and $\Psi$ are invariant under the action of $\mathbb{Z}^N$ and $\hat{m}$ is equivariant with respect to the action of $\mathbb{Z}^N$, we may apply suitable translations so that the sequence $\{x_n\}_n$ is bounded in $\mathbb{Z}^N$. Then, passing to a further subsequence, we have
    $$
        \hat{u}_n\rightharpoonup \hat{u}\in E,\quad \hat{v}_n\rightharpoonup \hat{v}\in E,\quad \hat{u}\neq\hat{v},\quad \Phi'(\hat{u})=\Phi'(\hat{v})=0,
    $$
    and $\|\hat{u}_n\|\to \alpha$, $\|\hat{v}_n\|\to \beta$ as $n\to\infty$. Let $c:=\sqrt[p]{pb}$, it follows from \cref{lem: Phi_BoundedBelow} that $\alpha,\beta\in[c,\nu_d]$. 
    
    We first consider the case where either $\hat{u}$ or $\hat{v}$ is $0$. Without loss of generality, assume that $\hat{u}=0$. Then $\hat{v}\neq0$, $\hat{v}\in\mathcal{N}$ and hence
    $$
    \varlimsup_{n\to\infty}\|u_n-v_n\|=\varlimsup_{n\to\infty}\left\|\frac{\hat{u}_n}{\|\hat{u}_n\|}-\frac{\hat{v}_n}{\|\hat{v}_n\|}\right\|\geq\|\frac{\hat{v}}{\beta}\|\geq\frac{c}{\nu_d},
    $$
    where the first inequality follows from the weak convergence of $\{\hat{v}_n\}_n$.

    It remains to consider the case where both $\hat{u}$ and $\hat{v}$ are not $0$. It follows that $\hat{u},\hat{v}\in\mathcal{N}$ and 
    $$
    u:=\frac{\hat{u}}{\|\hat{u}\|}\in S,\quad v:=\frac{\hat{v}}{\|\hat{v}\|}\in S,\quad u\neq v.
    $$
    Then by the weak convergence of $\{\hat{u}_n\}_n$ and $\{\hat{v}_n\}_n$ we have
    $$
    \varlimsup_{n\to\infty}\|u_n-v_n\|=\varlimsup_{n\to\infty}\left\|\frac{\hat{u}_n}{\|\hat{u}_n\|}-\frac{\hat{v}_n}{\|\hat{v}_n\|}\right\|\geq\|\frac{\hat{u}}{\alpha}-\frac{\hat{v}}{\beta}\|.
    $$
    Since ${1}/{\alpha},{1}/{\beta}\in[1/\nu_d,1/c]$ and $\hat{u},\hat{v}\in\Phi^d\cap \Lambda$ by \cref{lem: LimitOfPS_Remians}, it suffices to show that there exists a constant $\mu_d$ such that for all distinct $f,g\in\Phi^d\cap \Lambda$,
    $$
    \|\xi f-\eta g\|\geq\mu_d>0,\quad\forall\xi,\eta\in[1/\nu_d,1/c].
    $$
   Suppose, for contradiction, that no such $\mu_d$ exists. Then there exist sequences $\{\xi_n\}_n,\{\eta_n\}_n\subset[1/\nu_d,1/c]$ and $\{f_n\}_n,\{g_n\}_n\subset \Phi^d\cap \Lambda$ such that 
    $$
    \text{$f_n\neq g_n$ for all $n\in\mathbb{N}$\quad and \quad$\|\xi_n f_n-\eta_n g_n\|\to 0$ as $n\to\infty$.}
    $$
    Let $\check{f}_n:=m^{-1}(f_n)$ and $\check{g}_n:=m^{-1}(g_n)$. If $\|f_n-g_n\|_q\to 0$ as $n\to\infty$, then the same arguement as in \textbf{Case 1} implies that $\|f_n-g_n\|\to0$ as $n\to\infty$. It follows that $\|\check{f}_n-\check{g}_n\|\to 0$ as $n\to\infty$, which contradicts the result in \cref{lem: DiscretenessOfSolutions}. Therefore, we must have $\varlimsup\limits_{n\to\infty}\|{f}_n-{g}_n\|_q>0$. Proceeding as before, we may assume that 
    $$
        f_n\rightharpoonup f\in E,\quad g_n\rightharpoonup {g}\in E,\quad f\neq g,\quad \Phi'(f)=\Phi'(g)=0,
    $$
    and $\xi_n\to \xi$, $\eta_n\to\eta$ as $n\to\infty$, where $\xi,\eta\in[1/\nu_d,1/c]$. If either $f$ or $g$ is $0$, then
    $$
    \varlimsup_{n\to\infty}\|\xi_n{f}_n-\eta_n{g}_n\|\geq\max\{\|\xi f\|,\|\eta g\|\}\geq\frac{c}{\nu_d},
    $$
    which is a contradiction. Thus, both $f$ and $g$ must be nonzero. It then follows that
    $$
    \|\xi{f}-\eta{g}\|\leq \varliminf_{n\to\infty}\|\xi_n{f}_n-\eta_n{g}_n\|=0.
    $$
    Since $\xi,\eta>0$ and $f,g\in\mathcal{N}$, \cref{lem: UniquenessOft_u_G} implies that $f=g$, which is a contradiction. This completes the proof.
\end{proof}

\textbf{Acknowledgements.} 
I thank Bobo Hua for his continuous support and invaluable suggestions, which have greatly improved both the quality and the clarity of this manuscript. I also thank Zhiqiang Wang and Huyuan Chen for their helpful discussions and guidance during this research.

\bibliographystyle{plain}
\bibliography{main}

\noindent Xinrong Zhao, xrzhao24@m.fudan.edu.cn\\
\emph{School of Mathematical Sciences, Fudan University, Shanghai, 200433, P.R. China}\\[-8pt]
\end{document}